\documentclass[letterpaper,leqno,11pt,twoside]{amsart}
\usepackage{amsmath,amsthm,amsfonts,amssymb,euscript,mathrsfs,graphics,color,latexsym,marginnote}
\usepackage[dvips]{graphicx}
\usepackage[margin=0.8in]{geometry}
\usepackage[toc,page]{appendix}
\usepackage{relsize}
\usepackage[shortlabels]{enumitem}
\usepackage[all]{xy}
\usepackage{dsfont}

\usepackage[dvipsnames]{xcolor}

\usepackage{hyperref}
\hypersetup{colorlinks=true, pdfstartview=FitV,linkcolor=blue!70!black,citecolor=red!70!black, urlcolor=green!60!black}
\definecolor{labelkey}{rgb}{0.6,0,0} 

\numberwithin{equation}{section}

\newtheorem{theorem}{Theorem}[section]

\newtheorem{lemma}[theorem]{Lemma}
\newtheorem{proposition}[theorem]{Proposition}

\theoremstyle{definition}
\newtheorem{definition}{Definition}[section]
\theoremstyle{remark}
\newtheorem{remark}{Remark}[section]

\providecommand{\abs}[1]{\left\vert#1\right\vert}
\providecommand{\babs}[1]{\big\vert#1\big\vert}
\providecommand{\Babs}[1]{\Big\vert#1\Big\vert}
\providecommand{\oabs}[1]{\vert#1\vert}
\providecommand{\nm}[1]{\left\Vert#1\right\Vert}

\providecommand{\bnm}[1]{\big\Vert#1\big\Vert}

\def\ud{\mathrm{d}}
\def\dt{\partial_t}

\def\ls{\lesssim}
\def\gs{\gtrsim}

\def\rt{\rightarrow}
\def\r{\mathbb{R}}
\def\no{\nonumber}
\def\ue{e}
\def\ui{\mathrm{i}}

\newcommand{\ep}{\varepsilon}

\def\R{\mathbb{R}}
\def\C{\mathbb{C}}
\def\Z{\mathbb{Z}}
\def\N{\mathbb{N}}

\def\T{\mathbb{T}}

\def\ah{u}

\def\uh{u}

\def\wah{\widehat{\ah}}

\def\pe{\psi}

\def\ph{\psi^h}
\def\wph{\widehat{\ph}}
\def\qe{\phi}

\def\qh{\phi^h}
\def\wqh{\widehat{\qh}}

\def\id{\mathds{1}}
\def\rr{\mathcal{R}}

\def\k{\kappa}

\def\l{\ell}

\def\D{\Delta}

\def\DD{\Delta_{{\rm d}}}
\def\CC{\mathcal{C}}

\def\MM{\mathbf{M}}
\def\EE{\mathbf{E}}

\def\Ps{\mathsf{P}}

\newcommand{\qtq}[1]{\quad\text{#1}\quad}
\let\Re=\undefined
\DeclareMathOperator{\Re}{Re}
\let\Im=\undefined
\DeclareMathOperator{\Im}{Im}
\DeclareMathOperator{\supp}{supp}

\allowdisplaybreaks

\begin{document}

\title{Lattice Approximations to NLS}

\author[Zhimeng Ouyang]{Zhimeng Ouyang$^*$}
\address{Department of Mathematics, University of Chicago, Chicago, IL 60637, USA}
\email{ouyangzm9386@uchicago.edu}
\thanks{$^*$The author is partially supported by NSF grant DMS-2202824.}

\begin{abstract}
In this paper, we prove that 
solutions of the discrete NLS lattice model for $L^2$ initial data with double frequency components converge to solutions of a coupled system of cubic NLS.
\end{abstract}

\keywords{lattice model, continuum limit, almost conservation law, Strichartz estimate, nonlinear Schr\"odinger equation.}

\maketitle

\setcounter{tocdepth}{1}
\makeatletter
\def\l@subsection{\@tocline{2}{0pt}{2.5pc}{5pc}{}}
\makeatother
\tableofcontents

\section{Introduction}

Lattice models play a pivotal role in the investigation of microscopic multi-particle systems, with their continuum limits forming the foundation of the macroscopic effective theory. These models have found broad applications in condensed matter physics and numerical analysis.

As a prototypical example, we consider 
the $1$-d cubic discrete nonlinear Schr\"odinger equation for the evolution of a field $\ah_n(t)=\ah(t,n): \R\times\Z\rt\C$ 
\begin{align}\label{1-c.DNLS} \tag{DNLS} 
\ui\dt\ah_n &= -\big(\ah_{n+1}-2\ah_n+\ah_{n-1}\big)
\pm2|\ah_n|^2\ah_n \\
&=: -(\DD\ah)_n + \CC_n[\ah]  \no,
\end{align}
where $+$ corresponds to the defocusing case and $-$ corresponds to the focusing case.

To build the convergence from \eqref{1-c.DNLS} to its continuum counterpart, we introduce 
a small parameter $0<h \leq1$ as the length scale associated to the continuum approximation.
As $h\rt0$, 
ansatz of the form  
\begin{equation} \label{h-scaling}
    \ah_n(0) = h\pe_0(hn), \qquad
    \ah_n(t) = h\pe(h^2 t,hn)
\end{equation}
sends \eqref{1-c.DNLS} to the $1$-d cubic nonlinear Schr\"odinger equation 
\begin{equation}\label{1-c.NLS} \tag{NLS}
\ui\dt q = -\Delta q \pm2|q|^2 q ,
\end{equation}
provided that the solution is sufficiently smooth.
Therefore, the prefactor $h$ appearing in \eqref{h-scaling} indicates the precise scaling that balances the dispersion and nonlinear effects, which yields a continuum limit as $h\rt0$.

% ...

Our goal is to rigorously justify the continuum limit from \eqref{1-c.DNLS} to \eqref{1-c.NLS} for $L^2$ initial data with double frequency components.
Specifically, we will show that \emph{low-regularity} initial data of the form 
\begin{equation*}\label{dumb double data}
    \uh_n(0) = h \psi_0 (hn) +  (-1)^n h \phi_0 (hn),
\end{equation*}
which combines both slowly varying (low-frequency) and rapidly oscillating (high-frequency) parts, leads to solutions to a \emph{coupled} system of cubic NLS:
\begin{align}
    \ui\dt\pe&=-\Delta\pe\pm2\big(\abs{\pe}^2+2\abs{\qe}^2\big)\pe , 
    \label{NLS_low}\tag{${\rm NLS}_\psi$}\\
    \ui\dt\qe&=\Delta\qe\pm2\big(\abs{\qe}^2+2\abs{\pe}^2\big)\qe . 
    \label{NLS_high}\tag{${\rm NLS}_\phi$}
\end{align}

\medskip

\subsection{Problem Formulation}

In order to rigorously formulate our result, we first need to explain how we pass between functions on the real line $\R$ and the lattice $\Z$.

\smallskip
\subsubsection{Initial Data Sampling}

Given $\psi_0, \phi_0 \in L^2(\R)$.
We assemble the initial data with a mild smoothing (frequency truncation) before sampling:
\begin{align}\label{E:initial data}
\uh_n (0) = h \big[P_{\leq N} \psi_0\big] (hn) +  (-1)^n h \big[P_{\leq N} \phi_0\big] (hn).
\end{align}
Here $P_{\leq N}$ denotes a smooth projection which eliminates frequencies $|\xi|> 2N:= h^{-1+\gamma}$ for some $0<\gamma<1$ arbitrarily small.
Note that $N\to\infty$ as $h\to 0$, thus revealing the full irregularity of the initial data.

On the Fourier side for one of the $2\pi$-period $\theta\in\bigl[-\frac{\pi}{2},\frac{3\pi}{2}\bigr]$, we have (see \eqref{E:initial data''}) 
\begin{align} \label{E:initial data'}
    \widehat{\uh}(0,\theta) = \widehat{P_{\leq N} \pe_0}\big(\tfrac{\theta}{h}\big) + \widehat{P_{\leq N} \qe_0}\big(\tfrac{\theta-\pi}{h}\big) 
\end{align}
and so is supported near $0$ and $\pi$ for $\abs{\theta}\leq 2hN= h^{\gamma}<\frac{\pi}{2}$ and $\abs{\theta-\pi}\leq 2hN= h^{\gamma}<\frac{\pi}{2}$ (mod $2\pi$).
For the record,
we may set $0<h\leq h_0$ where $h_0\leq1$ is a small constant 
such that $\widehat{\uh}(0,\theta)$ is supported away from $\theta=\pm\frac{\pi}{2}$. 
As we wish to study the $h\rt0$ limit and so will only look at the scenario where $0<h\ll1$ is sufficiently small, this requirement is actually of no real consequence.

\smallskip
\subsubsection{Solution Reconstruction}

To compare the solution of \eqref{1-c.DNLS} to that of \eqref{NLS_low}-\eqref{NLS_high}, we generate two continuous functions $\ph(t,x)$ and $\qh(t,x)$ on the real line, from the single lattice function $\uh_n(h^{-2}t)$ via
\begin{align}
    \widehat{\psi^h}(t,\xi)&:= \widehat{\uh}(h^{-2}t, h\xi) \,\id_{(-\frac{\pi}2,\frac\pi 2)}(h\xi), \label{psi^h'}\\
    \widehat{\phi^h}(t,\xi)&:= e^{4\ui h^{-2}t}\, \widehat{\uh}(h^{-2}t, h\xi +\pi ) \,\id_{(-\frac{\pi}2,\frac\pi 2)}(h\xi). \label{phi^h'}
\end{align}
Essentially, here we split the Fourier transform $\widehat\uh$ into two pieces using a sharp cutoff to either semicircle; $\widehat{\psi^h}$ corresponds to the one around $0$ and $\widehat{\phi^h}$ corresponds to the other around $\pi$.
Note that both the reconstructed solutions $\psi^h$ and $\phi^h$ are still band-limited functions with $\supp(\widehat{\ph}),\,\supp(\widehat{\qh})\subseteq\bigl[-\frac{\pi}{2h},\frac{\pi}{2h}\bigr]$.

On the physical side, we write the equivalent expressions:
\begin{align}
    \psi^h(t,x) &:= h^{-1} \int_{|\theta|<\frac{\pi}2} e^{\ui \frac{x}{h}\theta}\,\widehat{\uh}(h^{-2}t,\theta) \,\tfrac{\ud\theta}{2\pi} 
    = \rr\big[\uh_n(h^{-2}t)\big](x) ,\label{psi^h}\\
    \phi^h(t,x) &:= h^{-1} e^{4\ui h^{-2}t} \int_{|\theta-\pi|<\frac{\pi}2} e^{\ui \frac{x}{h}(\theta-\pi)}\,\widehat{\uh}(h^{-2}t,\theta) \,\tfrac{\ud\theta}{2\pi} 
    = e^{4\ui h^{-2}t} \,\rr\big[(-1)^n\uh_n(h^{-2}t)\big](x) ,\label{phi^h}
\end{align}
through which we introduce the reconstruction operator $\rr:\ell^2_n(\Z)\rt L^2_x(\R)$ defined by
\begin{align}\label{R-def}
[\mathcal R a] (x) := h^{-1}\! \int_{-\frac\pi2}^{\frac\pi2} e^{\ui\frac{x}{h}\theta} \widehat a(\theta)\tfrac{\ud\theta}{2\pi} = \text{P.V.}\sum_{n\in\Z} \frac{\sin\!\left(\tfrac{\pi}{2h}(x-hn)\right)}{\pi(x-hn)}\cdot a_n
\end{align}
which embodies the spirit of the Nyquist–Shannon sampling theorem.

Restricted on the lattice, we can reverse the passage from $\uh_n$ to $\big(\ph,\qh\big)$ via
\begin{align} \label{u_n-ph,qh}
    \uh_n(t) = h\ph(h^2t,hn)+e^{-4\ui t}(-1)^n h\qh(h^2t,hn),
\end{align}
or equivalently on the Fourier side, 
\begin{align} \label{u_n-ph,qh'}
    \widehat{\uh}(t,\theta) \,\id_{[-\frac{\pi}{2},\frac{3\pi}{2}]}(\theta) 
    = \widehat{\ph}\big(h^2t,\tfrac{\theta}{h}\big) + e^{-4\ui t}\,\widehat{\qh}\big(h^2t,\tfrac{\theta-\pi}{h}\big).
\end{align}

Note also that the choice \eqref{E:initial data} of initial data corresponds to 
\begin{align} \label{ph,qh_0}
    \big[\ph,\qh\big](0,x) = \rr\big[\uh_n(0),(-1)^n\uh_n(0)\big](x)
    = \big[P_{\leq N} \psi_0, P_{\leq N} \phi_0\big](x)
    = P_{\leq h^{-1+\gamma}/2} \big[\psi_0, \phi_0\big](x) .
\end{align}

\medskip

\subsection{Main Result}

Our main result in this paper is the following:

\begin{theorem}\label{Thm:main}
Fix $(\psi_0,\phi_0)\in L^2(\R)$ and let $(\psi,\phi)\in \big(C_tL^2_x\cap L_{t,loc}^6L_x^6\big)(\R\times\R) $ denote the unique global solution of the coupled system \eqref{NLS_low}-\eqref{NLS_high} with this initial data.
Given $0<h\leq h_0\ll1$ sufficiently small, let $\uh_n(t) \in C_t\ell_n^2(\R\times\Z)$ be the global solution to \eqref{1-c.DNLS} with initial data $\uh_n(0)$ specified by \eqref{E:initial data} and let $\psi^h,\phi^h:\R\times\R\to \C$ be the corresponding continuum representatives of this solution built via the reconstruction formulas \eqref{psi^h},\eqref{phi^h}. 

Then for any $T>0$, as $h\to 0$ we have the continuum limit: 
the low- and high-frequency components of the reconstructed discrete solution converge to their respective continuum counterparts 
\begin{align}\label{E:T:main}
    \psi^h(t,x)\rt\psi(t,x) \qtq{and} \phi^h(t,x)\rt\phi(t,x)
    \quad \text{in \,$C_tL^2_x([-T,T]\times\r)$};
\end{align}
moreover, this yields the long-wave limit of the discrete model: $(\psi,\phi)$ describes the small-$h$ behavior of $\uh_n$ in the sense that 
\begin{align}\label{E:T:main-2}
   h^{\frac{1}{2}} \left\| h^{-1}\uh_n(h^{-2}t) - \pe(t,hn) - e^{-4\ui h^{-2}t}(-1)^n \qe(t,hn) \right\|_{\l_n^2} \rt 0
   \quad \text{uniformly for \,$|t|\leq T$}.
\end{align}
\end{theorem}

\begin{remark}
Some remarks on the main theorem:
\begin{enumerate}
    \item 
    This result suggests that a sole NLS does not suffice to encapsulate the lattice model  dynamics in such a low-regularity setting, which is reminiscent of the thermal equilibrium state.
    In particular, the high-frequency component does not converge directly to NLS without appropriate transformations.
    \item 
    Compared to the integrable discretization Ablowitz--Ladik system (AL) which converges to a decoupled NLS (cf. \cite{AA021}), 
    the presence of interactions of the two frequency components for \eqref{1-c.DNLS} is exactly due to the different structure of nonlinearity.
    This highlights the discrepancy between the integrable and non-integrable systems. 
    \item 
    Our result readily implies the simpler case of single frequency component, where the limit should solve the single \eqref{1-c.NLS}.
\end{enumerate}
\end{remark}

To recap, our problem formulation and main result can be illustrated as follows:  

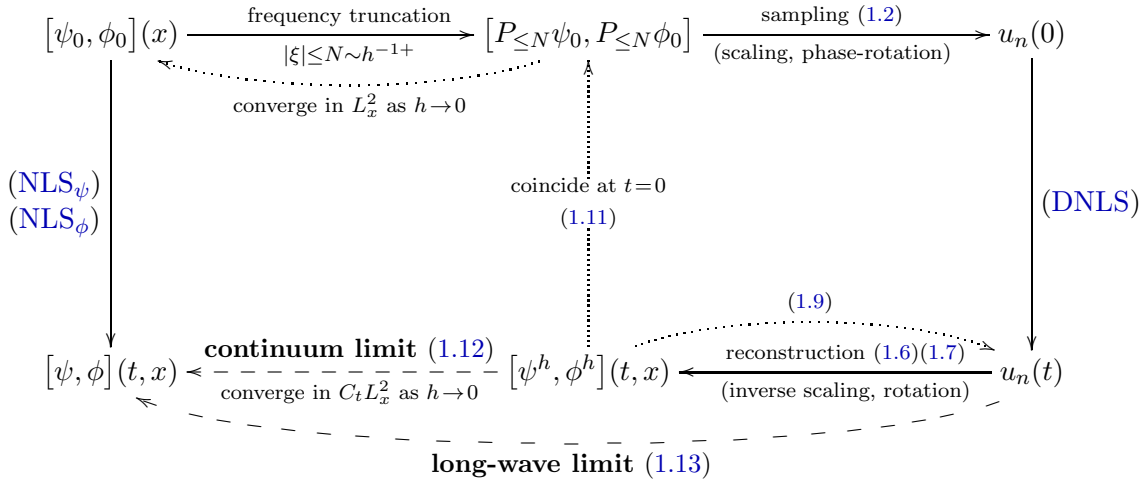
\begin{equation*}
\xymatrix@C=9pc@R=9pc{
\big[\pe_0,\qe_0\big](x) \ar[r]^{\txt{\scriptsize frequency truncation}}_{\abs{\xi}\leq N\sim h^{-1+}} \ar[d]_{\txt{\eqref{NLS_low}\\\eqref{NLS_high}}} & 
\big[P_{\leq N} \psi_0,P_{\leq N} \phi_0\big] \ar[r]^(.55){\txt{\scriptsize sampling \eqref{E:initial data}}}_(.55){\txt{\scriptsize (scaling, phase-rotation)}} \ar@{.>}@(dl,dr)[l]^(.5){\txt{\scriptsize converge in $L^2_x$ as $h\!\rt\!0$}} &
\uh_n(0) \ar[d]^{\txt{\eqref{1-c.DNLS}}} \\
\big[\pe,\qe\big](t,x) & 
\big[\ph,\qh\big](t,x) \ar@{-->}[l]_(.5){\txt{\small\textbf{continuum limit \eqref{E:T:main}}}}^(.5){\txt{\scriptsize converge in $C_tL^2_x$ as $h\!\rt\!0$}} \ar@{.>}[u]|-{\txt{\scriptsize coincide at $t\!=\!0$\\\scriptsize\eqref{ph,qh_0}}} \ar@{.>}@(ur,ul)[r]^(.5){\txt{\scriptsize \eqref{u_n-ph,qh}}} & 
\uh_n(t) \ar[l]_(.42){\txt{\scriptsize reconstruction \eqref{psi^h}\eqref{phi^h}}}^(.42){\txt{\scriptsize (inverse scaling, rotation)}} \ar@{-->}@(d,d)[ll]^(.5){\txt{\small\textbf{long-wave limit \eqref{E:T:main-2}}}}
}
\end{equation*}

\medskip

\subsection{Method}

Our general strategy synthesizes compactness, almost conservation laws, and Strichartz-based techniques.
In particular, precompactness follows from three properties: uniform boundedness, equicontinuity, and tightness.

Complete integrability of the Ablowitz–Ladik system (AL) plays a key role in the proof of \cite{AA021}, which relies on various conserved quantities and the generating function of a family of infinitely many conservation laws. 
This is very different from what we shall be doing:
Unlike AL, the general DNLS is not integrable and lacks enough conserved quantities, so we have to exploit other new ingredients to close a compactness-uniqueness scheme; 
indeed, much of our analysis is based on the method of almost conservation laws combined with Strichartz estimates, especially in obtaining spatial equicontinuity and in controlling the frequency localization for the double frequency components setting.

It is reasonable to imagine that our approach could be expanded to treat more continuum limit problems for non-integrable lattice models.

\bigskip

\section{Preliminaries}

\subsection{Notation and Convention}

Throughout this paper, $C$ will denote a universal constant that does not depend on $h$, and which may vary from one line to another.
We write $X \lesssim Y$ or $Y\gtrsim X$ whenever $X\leq CY$ for some constant $C>0$.  We write $X\simeq Y$ to mean $X\lesssim Y$ and $Y\lesssim X$.  
We say that $X\ll Y$ if $X\leq cY$ for some small constant $c>0$, again not depending on $h$.
We may use $C \gg 1$ to denote various large finite constants, and $0 < c \ll 1$ to denote various small constants.
If $C$ or $c$ depends on some additional parameters, we will indicate this with subscripts.
We use also the notation $X+ := X+\ep$ for some arbitrarily small $0 < \ep \ll 1$ and similarly $X- := X-\ep$.

Our convention for the Fourier transform on $\r$ is as follows:
\begin{align}\label{FT-c}
    \widehat{f}(\xi) =[\mathscr{F}\!f](\xi) =\int_\R f(x) e^{-\ui x\xi}\,\ud x \qquad\text{so that}\qquad f(x) =\big[\mathscr{F}^{-1}\widehat{f}\,\big](x) =\int_\R\widehat{f}(\xi) e^{\ui x\xi}\,\tfrac{\ud\xi}{2\pi},
\end{align}
while for the Fourier series (in the discrete case), we employ 
\begin{align}\label{FS-d}
\widehat a(\theta) =[\mathscr{F}_{\!\rm d}a](\theta) = \sum_{n\in \Z} a_n e^{-\ui n\theta} \qquad\text{so that}\qquad a_n =\big[\mathscr{F}_{\!\rm d}^{-1}\widehat{a}\,\big](n) = \int_{-\pi}^{\pi} \widehat a(\theta) e^{\ui n\theta}\, \tfrac{\ud\theta}{2\pi}.
\end{align}
Here the frequency variable $\theta\in\R/2\pi\Z$; 
we will sometimes use the shorthand $\T:=\R/2\pi\Z$ for this periodic domain (`circle' in $1$-d).

With these definitions, the Plancherel/Parseval identities read
\begin{align}\label{Plancherel}
    \int_\R \bigl|f(x)\bigr|^2\, \ud x = \int_\R \bigl|\widehat{f}(\xi)\bigr|^2\,\tfrac{\ud\xi}{2\pi} \qquad\text{and}\qquad
    		\sum_{n\in \Z} |a_n|^2= \int_{-\pi}^{\pi} \bigl|\widehat a(\theta)\bigr|^2\,\tfrac{\ud\theta}{2\pi},
\end{align}
and equivalently, 
\begin{align}\label{Parseval}
    \int_\R f(x)\overline{g(x)}\, \ud x = \int_\R \widehat{f}(\xi)\overline{\widehat{g}(\xi)}\,\tfrac{\ud\xi}{2\pi} \qquad\text{and}\qquad
    		\sum_{n\in \Z} a_n\overline{b_n} = \int_{-\pi}^{\pi} \widehat{a}(\theta) \overline{\widehat{b}(\theta)}\,\tfrac{\ud\theta}{2\pi}.
\end{align}
In addition, we write the convolution property for $\mathscr{F}_{\!\rm d}$ as a paraproduct: 
\begin{align} \label{convolution-paraproduct}
    \widehat{ab}(\theta) = \int_{-\pi}^{\pi} \widehat a(\theta-\eta) \widehat b(\eta)\, \tfrac{\ud\eta}{2\pi} = \int_{\theta=\theta_1+\theta_2} \widehat a(\theta_1) \widehat b(\theta_2) ,
\end{align}
and consequently, 
\begin{align} \label{Parseval-paraproduct}
    \sum_{n\in\Z} (a_1)_n (a_2)_n\cdots (a_k)_n =  \int_{\theta_1+\theta_2+\cdots+\theta_k=0} \widehat a_1(\theta_1) \widehat a_2(\theta_2)\cdots \widehat a_k(\theta_k),
\end{align}
where $\int_{\theta_1+\theta_2+\cdots+\theta_k=0}$ denotes integration with respect to the hyperplane’s measure $2\pi\delta_0(\theta_1+\theta_2+\cdots+\theta_k)\tfrac{\ud\theta_1}{2\pi}\tfrac{\ud\theta_2}{2\pi}\cdots\tfrac{\ud\theta_k}{2\pi}$ with $\delta_0$ being the one-dimensional Dirac delta function.

Let $\varphi$ be a smooth radial bump function supported in the ball $|\xi| \leq 2$ such that $\varphi(\xi)=1$ for all $|\xi| \leq 1$. For each dyadic number $N,M \in 2^\Z$, we define the Littlewood--Paley projection operators
\begin{align*}
&\widehat{P_{\leq N}f}(\xi) :=  \varphi\bigl(\tfrac{\xi}N\bigr)\widehat f (\xi), \qquad 
P_{N} :=  P_{\leq N} - P_{\leq N/2} ,\qquad
P_{> N} :=  1-P_{\leq N} ,\\
&P_{M<\cdot\,\leq N} :=  P_{\leq N} - P_{\leq M} = \sum_{M<N'\leq N} P_{N'} .
\end{align*}
Any sum over the dummy variable $N$, $N'$ or $M$ is understood to be over dyadic numbers unless otherwise specified.
We may also write $f_{\leq N}$ for $P_{\leq N} f$ and similarly for the other operators.
Like all Fourier multiplier operators, the Littlewood–Paley operators commute with differential operators such as $\ui\partial_t + \Delta$, as well as with the propagator $e^{\ui t\Delta}$.
They are bounded (uniformly in $N$) on $L^p(\R)$ for every $1 \leq p \leq\infty$.

We will also frequently use other variants of $P$ to generally denote a Fourier truncation operator (to be specified).
In particular, 
we introduce the Fourier projection adapted to our discrete setting which treats frequencies near $\theta=0$ and $\theta=\pi$ simultaneously:

\begin{definition} \label{Def:Ps}
Given a parameter $0<\lambda<1$, 
let $\Ps_{\!<\lambda}$ be the sharp Fourier cutoff (for discrete functions) defined via
\begin{align} \label{Ps-def}
\widehat {\Ps_{\!<\lambda\,} a} (\theta) := \id_{\mathsf G_\lambda\!}(\theta) \,\widehat a(\theta) \qtq{with} \mathsf G_\lambda:=\bigl\{\theta\in\R/2\pi\Z : \abs{\sin(\theta)} < \lambda\bigr\} .
\end{align}
(Since $\frac{\sin(\theta)}{\theta}\in\left[\tfrac{2}{\pi},1\right]$ i.e. $\sin(\theta)\simeq\theta$ for $\abs{\theta}\leq\frac{\pi}{2}$, we see $\abs{\sin(\theta)} < \lambda$ $\Leftrightarrow$ $\abs{\theta-n\pi}<c\lambda$ for some $n\in\Z$ and $c>0$.)

Define also $\Ps_{\geq\lambda} := 1-\Ps_{\!<\lambda}$ with the Fourier symbol $\id_{\mathsf G_\lambda^c\!}(\theta)$, the indicator function to the complement set $\mathsf G_\lambda^c$.
\end{definition}

\medskip

\subsection{Norms}

In this paper, we will employ the Lebesgue spaces $L^p$/$\l^p$ as well as the spacetime spaces $L_t^q L_x^p$/$L^q_t\l^p_n$ for either continuum or discrete version, equipped with the norms:
\begin{align*}
    \|F\|_{L_t^q L_x^p(\R\times\R)} &:=\bigg[\int_{\R}\Bigl(\int_{\R} \big|F(t,x)\big|^p \,\ud x \Bigr)^{\frac{q}{p}} \,\ud t\bigg]^{\frac{1}{q}},\\
    \nm{\ah}_{L^q_t\l^p_n(\R\times\Z)} &:= \bigg[\int_{\R}\Big(\sum_{n\in\Z}\babs{\ah_n(t)}^p\Big)^{\frac{q}{p}}\ud t\bigg]^{\frac{1}{q}},
\end{align*}
with the usual modifications when $q$ or $p$ is $\infty$, or when the domain is replaced by some smaller subset.
When $q=p$ we abbreviate $L_t^q L_x^p$ by $L^q_{t,x}$.

Unlike general $L^p$ spaces which do not have the inclusion relation,
there is a simple embedding result for the $\l^p$ spaces of infinite sequences/discrete functions:

\begin{lemma}[Embedding for $\l^p$ spaces] \label{Lem:l^p-embedding}
The $\l^p$ spaces are increasing in $p\in[1,\infty]$, with the inclusion operator being continuous: 
For $1\leq p<q\leq\infty$, one has $\l^p\hookrightarrow\l^q$ with
\begin{align} \label{l^p-embedding}
    \nm{a}_{\l^q} \leq \nm{a}_{\l^p} 
    \quad \text{for any }\: a\in\l^p.
\end{align}
\end{lemma}

\begin{proof}
Since the inequality above and the $\l^p$-norms are homogeneous in all entries $a_n$, it is sufficient to prove \eqref{l^p-embedding} under the normalization $\nm{a}_{\l^p}=1$.
In this case, $\nm{a}_{\l^p}=1$ implies $\abs{a_n}\leq1$ for all $n$, 
which means $\nm{a}_{\l^\infty} \leq 1$.
For $q\in(\,p,\infty)$, one also has $\sum_n\abs{a_n}^q \leq \sum_n\abs{a_n}^p =1$. 
(Alternatively, by H\"older/interpolation, $\nm{a}_{\l^q} \leq \nm{a}_{\l^\infty}^{1-p/q} \nm{a}_{\l^p}^{p/q} \leq 1$.)
Together, we have shown that $\nm{a}_{\l^q} \leq 1 = \nm{a}_{\l^p}$ for $1\leq p<q\leq\infty$, as desired.
\end{proof}

\begin{remark}
This embedding result for the $\l^p$ spaces can be understood as a special case of the one for functions with finite frequency: 
Let $\Omega\subset\R^d$ be a compact subset and $L^p_\Omega(\R^d):=\big\{f\in L^p(\R^d): \supp(\widehat{f}\,)\subseteq\Omega\big\}$.
Then for $1\leq p<q\leq\infty$, one has $L^p_\Omega(\R^d)\hookrightarrow L^q_\Omega(\R^d)$, with $\nm{f}_{L^q} \ls \nm{f}_{L^p}$ for any $f\in L^p_\Omega(\R^d)$.
\end{remark}

\medskip

\subsection{Passage Between Lattice and Continuum Functions}

Let us now discuss the passage between functions on $\R$ and those on the $h\Z$ lattice.  We need to examine both directions: the initial data $\uh_n(0)$ is constructed from the data $(\psi_0,\phi_0)$ for \eqref{NLS_low}-\eqref{NLS_high} via \eqref{E:initial data}; the resulting discrete solutions $\uh_n(t)$ of \eqref{1-c.DNLS} are then transferred back to the real line via \eqref{psi^h},\eqref{phi^h}.

With our conventions for the Fourier transform/series, we record
the Poisson summation formula: 

\begin{lemma} \label{Poisson-summation}
For Schwartz-class functions $f\in\mathscr{S}(\R)$, we have
\begin{align*}
    \sum_{n\in\Z} hf(hn)e^{-\ui n\theta}=\sum_{m\in\Z}\widehat{f}\left(\tfrac{\theta+2\pi m}{h}\right) .
\end{align*}
In particular, for $f\in L^2(\R)$ with $\supp(\widehat f\,)\subseteq\bigl[-\frac\pi h,\frac \pi h\bigr]$ and $|\theta|\leq\pi$, 
\begin{align*}
    \sum_{n\in\Z}hf(hn)e^{-\ui n\theta}=\widehat{f}\big(\tfrac{\theta}{h}\big).
\end{align*}
\end{lemma}
This lemma allows us to build connections between the band-limited continuum function and its lattice sampling on the Fourier side:
For the construction of initial data given by \eqref{E:initial data} (with $2hN= h^{0+}<\frac{\pi}{2}$ so that $\supp(\widehat{P_{\leq N} \pe_0}),\,\supp(\widehat{P_{\leq N} \qe_0})\subseteq\bigl[-\frac{\pi}{2h},\frac{\pi}{2h}\bigr]$), we 
apply Lemma~\ref{Poisson-summation} to $P_{\leq N} \pe_0$ for $|\theta|\leq\frac{\pi}{2}$ and $P_{\leq N} \qe_0$ for $|\theta-\pi|\leq\frac{\pi}{2}$ and hence find 
\begin{align} \label{E:initial data''}
    \widehat{\uh}(0,\theta) \,\id_{[-\frac{\pi}{2},\frac{3\pi}{2}]}(\theta) 
    &= \sum_{n\in\Z}\uh_n(0)\ue^{-\ui n\theta} \,\id_{[-\frac{\pi}{2},\frac{3\pi}{2}]}(\theta) \\
    &=  \sum_{n\in\Z} h \big[P_{\leq N} \psi_0\big] (hn) \,\ue^{-\ui n\theta} \,\id_{[-\frac{\pi}{2},\frac{\pi}{2}]}(\theta)
    +  \sum_{n\in\Z} h \big[P_{\leq N} \phi_0\big] (hn) \,\ue^{-\ui n(\theta-\pi)} \,\id_{[\frac{\pi}{2},\frac{3\pi}{2}]}(\theta) \no\\
    &= \widehat{P_{\leq N} \pe_0}\big(\tfrac{\theta}{h}\big) + \widehat{P_{\leq N} \qe_0}\big(\tfrac{\theta-\pi}{h}\big), \no
\end{align}
thus justifying \eqref{E:initial data'}.
Also, 
using Plancherel \eqref{Plancherel} and recalling that $N\sim h^{-1+}$, we have the $L^2\rt\l^2$ boundedness of the initial data mapping 
\begin{align} \label{ah0-h'}
\nm{\ah(0)}_{\l_n^2}^2
= h \Big(\bnm{P_{\leq N} \psi_0}_{L_x^2}^2 + \bnm{P_{\leq N}\phi_0}_{L_x^2}^2 \Big) 
\leq h \big( \|\psi_0\|_{L^2}^2 + \|\phi_0\|_{L^2}^2 \big) ,
\end{align}
and subsequently
\begin{align} \label{ah0-h}
\lim_{h\to 0}\, h^{-1}\|\ah(0)\|_{\ell_n^2}^2 
= \lim_{h\to 0} \Big(\bnm{P_{\leq N} \psi_0}_{L_x^2}^2 + \bnm{P_{\leq N}\phi_0}_{L_x^2}^2 \Big)
=  \|\psi_0\|_{L^2}^2 + \|\phi_0\|_{L^2}^2,
\end{align}
which means $\|\ah(0)\|_{\ell_n^2} \sim O(h^{\frac{1}{2}})$ as $h\rt0$.

Another closely related result (which also follows directly from basics of the Fourier transform/series) is an isometry property regarding band-limited functions.
As an immediate corollary, we also derive a mathematical embodiment of the Nyquist–Shannon sampling theorem.

\begin{lemma}\label{L:sums to int}
If $f,g \in L^2(\R)$ satisfy $\supp(\widehat f\,),\supp(\widehat g)\subseteq\bigl[-\frac\pi h,\frac \pi h\bigr]$, then
\begin{align} \label{isometry}
\int_{\R} f(x) \,\overline{g(x)} \, \ud x 
\,=\, h\sum_{n\in\Z} f(hn) \,\overline{g(hn)}. 
\end{align}
In particular, for such a function $f$, $f\mapsto h^{\frac{1}{2}} f(hn)$ is a unitary map from $L^2_x(\R)$ to $\l^2_n(\Z)$: 
\begin{align} \label{unitary-map}
\int_{\R} \big|f(x)\big|^2 \, \ud x 
\,=\, h\sum_{n\in\Z} \big|f(hn)\big|^2 ; 
\end{align}
moreover, $f(x)$ for any $x\in\R$ is uniquely determined by its values at the lattice points $h\Z$ via 
\begin{align} \label{Shannon-sampling}
f(x) = \int_{-\frac\pi h}^{\frac\pi h} \widehat f(\xi)\, e^{\ui x\xi}\, \tfrac{\ud\xi}{2\pi} 
= \sum_{n\in\Z} \frac{\sin\!\big(\pi(x-hn)/h\big)}{\pi(x-hn)/h}\, f(hn).
\end{align}
\end{lemma}

\begin{proof}
The identity \eqref{isometry} is a consequence of the Parseval identity \eqref{Parseval} of the Fourier transform/series defined in \eqref{FT-c},\eqref{FS-d} with the change of variables $\theta=h\xi$ (and noting the compact support of $\widehat{f},\widehat{g}$):
\begin{equation} \label{isometry'}
\int_{\R} f(x) \,\overline{g(x)} \, \ud x 
= \int_{-\frac\pi h}^{\frac\pi h} \widehat f(\xi) \,\overline{\widehat g(\xi)} \,\tfrac{\ud\xi}{2\pi}
= h^{-1}\! \int_{-\pi}^{\pi} \widehat f(\tfrac{\theta}{h}) \,\overline{\widehat g(\tfrac{\theta}{h})} \,\tfrac{\ud\theta}{2\pi}
= h\sum_{n\in\Z} f(hn) \,\overline{g(hn)}. 
\end{equation}
Taking $g=f$, this immediately yields \eqref{unitary-map}. 
To deduce \eqref{Shannon-sampling}, we apply \eqref{isometry'} with $\widehat g(\xi) = e^{-\ui x\xi} \id_{[-\pi,\pi]}(h\xi)$ so that 
$g(y) = \int_{-\pi/h}^{\pi/h} e^{\ui (y-x)\xi} \,\tfrac{\ud\xi}{2\pi} = \frac{\sin\left(\pi(x-y)/h\right)}{\pi(x-y)}$.
\end{proof}

We remark that discrete functions can essentially be viewed as certain band-limited continuum functions.
This heuristics allows us to pass functions back and forth between $\Z$ and $\R$.
In particular, it gives rise to the
Reconstruction formula which transfers solutions of \eqref{1-c.DNLS} on the lattice to functions on the line:

\begin{definition}
Define the reconstruction operator $\rr:\ell^2_n(\Z)\rt L^2_x(\R)$ by
\begin{align} \label{R-def=}
[\rr a] (x) := h^{-1}\! \int_{-\frac\pi2}^{\frac\pi2} e^{\ui\frac{x}{h}\theta} \widehat a(\theta)\tfrac{\ud\theta}{2\pi} = \text{P.V.}\sum_{n\in\Z} \frac{\sin\!\left(\tfrac{\pi}{2h}(x-hn)\right)}{\pi(x-hn)}\cdot a_n .
\end{align}
Then the reconstructed solutions $\psi^h,\phi^h$ given in \eqref{psi^h},\eqref{phi^h} can be represented as
\begin{align} \label{psi^h phi^h}
    \psi^h(t,x) = \rr\big[\uh_n(h^{-2}t)\big](x)
    \quad\text{ and }\quad
    \phi^h(t,x) = e^{4\ui h^{-2}t} \,\rr\big[(-1)^n\uh_n(h^{-2}t)\big](x) .
\end{align}
\end{definition}

From the definition \eqref{R-def=} we have
$\supp(\widehat{\rr a})\subseteq\bigl[-\frac{\pi}{2h},\frac{\pi}{2h}\bigr]$
with $\widehat{\rr a}(\xi) = \id_{(-\frac{\pi}2,\frac\pi 2)} \widehat{a}(h\xi)$
and $[\rr a] (hn) = h^{-1} \big(P_{\abs{\theta}<\frac{\pi}{2}} a\big)_n $
(here $P_{\abs{\theta}<\frac{\pi}{2}}$ denotes the sharp projection to frequencies $\oabs{\theta}<\frac{\pi}{2}$).
These properties in turn characterize uniquely the function $[\rr a] (x)$ 
based on the Nyquist–Shannon sampling formula \eqref{Shannon-sampling}.

Furthermore, to later obtain the Strichartz bounds, we also need to generalize \eqref{unitary-map} to the $L_x^p(\R) \to \ell_n^p(\Z)$ property of the mapping for band-limited functions and understand the relation of the norms between our discrete solutions $\ah_n$ and the reconstructed solutions $\ph,\qh$. 
We refer to \cite{Plancherel.Polya1937} for the following lemma.


\begin{lemma} \label{Lem:norm-relation}
Fix $1<p<\infty$ and let $f:\R\to \C$ be a Schwartz function with $\supp(\widehat{f}\,)\subseteq\bigl[-\frac\pi h,\frac \pi h\bigr]$.  Then
\begin{align} \label{lp-Lp}
    \int_{\R}\big|f(x)\big|^p\, \ud x \,\simeq\,
    h\sum_{n\in\Z}\big|f(hn)\big|^p .
\end{align}
\end{lemma}

Combining this result with the $\l^p$-boundedness (for $1<p<\infty$) of sharp Fourier cutoffs, we find that
the operator $\mathcal R$ (defined in \eqref{R-def=}) is bounded from $\ell_n^p(\Z)$ to $L^p_x(\R)$ for $1<p<\infty$ with norm
\begin{align} \label{R:lp-Lp}
\|\mathcal R\|_{\ell_n^p(\Z)\to L^p_x(\R)}\lesssim h^{\frac1p -1}.
\end{align}
Consequently, under the relation \eqref{psi^h phi^h} and for $|t|\leq T$, we have for $1<p<\infty$ and $1\leq q\leq\infty$, 
\begin{align*}
\nm{\ah(h^{-2}t)}_{\l^p_n(\Z)} 
\simeq h^{1-\frac{1}{p}}
    \Big(\bnm{\ph(t)}_{L_x^p(\R)} + \bnm{\qh(t)}_{L_x^p(\R)} \Big) ,
\end{align*}
and 
\begin{align} \label{lqp-Lqp}
    \nm{\ah}_{L^q_{\tau}\l^p_n([-h^{-2}T,\,h^{-2}T]\times\Z)} 
    \simeq h^{1-\frac{1}{p}-\frac{2}{q}}
    \Big(\bnm{\ph}_{L_t^q L_x^p([-T,T]\times\R)} + \bnm{\qh}_{L_t^q L_x^p([-T,T]\times\R)} \Big) .
\end{align}
Here we set $\tau=h^{-2}t$; we will frequently use this shorthand in what follows.
Note also that the corresponding time interval for $\ah_n$ is stretched to a longer period of $h^{-2}T$,
and that the factor $h^{1-\frac{1}{p}-\frac{2}{q}}$ exactly comes from our particular space-time scaling for the continuum limit.

\medskip

\medskip

\subsection{Dispersive and Strichartz Estimates} \label{SubSec:Dispersive-Strichartz}
% ...

We start this subsection by reviewing some dispersive and Strichartz estimates for the continuum and discrete Schr\"odinger propagators.

The continuum Schr\"odinger propagator $e^{\ui t\D}$ corresponds to the dispersion relation $\omega(\xi)=-\xi^2$, which leads to the dispersive inequality 
\begin{equation*}
    \bnm{e^{\ui t\D}\psi_0}_{L^\infty_x} \ls \abs{t}^{-\frac{1}{2}} \nm{\psi_0}_{L^1_x} ,
\end{equation*}
as well as the Strichartz estimates:

\begin{lemma}[Linear Strichartz estimates for $e^{\ui t\D}$; \cite{Ginibre.Velo1992}] \label{Lem:l.c-Strichartz}
Let $(q,p)$ and $(\tilde q, \tilde p)$ be two Schr\"odinger-admissible Strichartz pairs satisfying
\begin{align*}
2\leq q,p, \tilde q, \tilde p\leq\infty \quad\text{and} \quad 
\tfrac{2}{q}+\tfrac{1}{p}=\tfrac{1}{2}=\tfrac{2}{\tilde q}+\tfrac{1}{\tilde p}.
\end{align*}
If $\psi$ solves the linear Schr\"odinger equation
\begin{align*}
\ui\dt\psi=-\D \psi+F
\quad\text{with} \quad
\psi(0)=\psi_0\in L^2_x(\R)
\end{align*}
on $I\times\R$ for some time interval $I\ni0$,
then
\begin{align*}
\nm{\psi}_{L_t^qL^p_x(I\times\R)}\ls\nm{\psi_0}_{L^2_x(\R)}+\nm{F}_{L_t^{\tilde q'}\!L_x^{\tilde p'}(I\times\R)}.
\end{align*}
\end{lemma}

In comparison,
the discrete Schr\"odinger propagator $e^{\ui t\DD}$ with the discrete Laplacian $(\DD\ah)_n:= \ah_{n+1} - 2\ah_n + \ah_{n-1}$ corresponds to 
\begin{align*}
    \omega_{{\rm d}}(\theta) = e^{\ui\theta}-2+e^{-\ui\theta}=2\cos(\theta)-2 = -4\sin^2(\tfrac{\theta}{2}) .
\end{align*}
Unlike in the continuum case, the discrete dispersion relation $\omega_{{\rm d}}(\theta)$ has inflection points at $\theta=\pm\frac{\pi}{2}$ (mod $2\pi$) where the dispersion is very weak, and this results in a slower dispersive decay (similar to the situation for the Airy propagator)
\begin{align*}
    \bnm{e^{\ui t\DD}\ah_0}_{\ell^\infty_n} \ls \abs{t}^{-\frac{1}{3}} \nm{\ah_0}_{\ell^1_n}.
\end{align*}
Then as shown in \cite[Theorem~3]{Stefanov.Kevrekidis2005}, the corresponding Strichartz estimates can be derived via the same techniques used in \cite{Keel.Tao1998}.

\begin{lemma}[Linear Strichartz estimates for $e^{\ui t\DD}$; \cite{Stefanov.Kevrekidis2005}] \label{Lem:l.d-Strichartz}
Let $(q,p)$ and $(\tilde q, \tilde p)$ be two pairs satisfying
\begin{align*}
2\leq q,p, \tilde q, \tilde p\leq\infty , \qquad 
\tfrac{1}{q}+\tfrac{1}{3p}\leq \tfrac{1}{6} \quad\text{and} \quad \tfrac{1}{\tilde q}+\tfrac{1}{3\tilde p}\leq \tfrac{1}{6}.
\end{align*}
If $\ah$ solves the linear discrete Schr\"odinger equation
\begin{align*}
\ui\dt\ah_n=-(\DD\ah)_n+F_n
\quad\text{with} \quad
\ah(0)=\ah_0\in \l^2_n(\Z)
\end{align*}
on $I\times\Z$ for some time interval $I\ni0$,
then
\begin{align*}
\nm{\ah}_{L_t^q\l^p_n(I\times\Z)}\ls\nm{\ah_{0}}_{\l^2_n(\Z)}+\nm{F}_{L_t^{\tilde q'}\!\l_n^{\tilde p'}(I\times\Z)}.
\end{align*}
\end{lemma}

However, if we project to frequencies away from the inflection points $\pm\frac{\pi}{2}$
using the projection operator $\Ps=\Ps_{\!<\lambda}$ (for some $0<\lambda<1$) defined in \eqref{Ps-def}, we may recover the same dispersive decay as in the continuum case: 
\begin{align*}
    \bnm{e^{\ui t\DD}\Ps}_{\ell^1\rt\ell^\infty} \sim \bnm{e^{\ui t\D}}_{L^1\rt L^\infty}
    \ls \abs{t}^{-\frac{1}{2}}.
\end{align*}
Consequently, we have the frequency-restricted Strichartz estimates for the discrete propagator with the same admissible pairs as for the continuum Schr\"odinger propagator.

\begin{lemma}[Linear Strichartz estimates for $e^{\ui t\DD}$ with frequency truncation]\label{P:loc Strichartz}
Let $(q,p)$ and $(\tilde q, \tilde p)$ be two pairs satisfying
\begin{align*}
2\leq q,p, \tilde q, \tilde p\leq\infty \quad\text{and} \quad 
\tfrac{2}{q}+\tfrac{1}{p}=\tfrac{1}{2}=\tfrac{2}{\tilde q}+\tfrac{1}{\tilde p}.
\end{align*}
The solution to the linear discrete Schr\"odinger equation
\begin{align*}
\ui\dt\ah_n=-(\DD\ah)_n+F_n
\quad\text{with} \quad
\ah(0)=\ah_0\in \l^2_n(\Z)
\end{align*}
on $I\times\Z$ for some time interval $I\ni0$
satisfies
\begin{align*}
\nm{\Ps\ah}_{L_t^q\l^p_n(I\times\Z)}\ls\nm{\Ps\ah_{0}}_{\l^2_n(\Z)}+\nm{\Ps F}_{L_t^{\tilde q'}\!\l_n^{\tilde p'}(I\times\Z)},
\end{align*}
where $\Ps=\Ps_{\!<\frac34}$ is the projection operator defined in \eqref{Ps-def}.
\end{lemma}

In the remainder of this section, we develop a key bilinear refinement of frequency-localized Strichartz estimates (which we will use in Section~\ref{Sec:Almost_CL}). 
This estimate expresses that there is little interaction between relatively high and low frequencies of the DNLS solutions, thus helping to control the movement of `energy' from high modes to low or vice versa.
We refer to the bilinear Strichartz estimates obtained in \cite{Koch.Tataru.Visan2014,Visan2007} (which build on earlier versions in \cite{Colliander.Keel.Staffilani.Takaoka.Tao2008,Bourgain1998,Bourgain1999,Ozawa.Tsutsumi1998})
for the continuum case.
We are not aware of any such results for the discrete Schr\"odinger propagator and so provide a proof in our specific context.

\begin{lemma}[(Frequency-localized) Bilinear Strichartz estimates for $e^{\ui t\DD}$] \label{Lem:bilinear-Strichartz}
Let $a,b\in\l^2_n(\Z)$ and let $a_K,b_L$ denote the Littlewood-Paley pieces supported on frequencies $\theta\in\R/2\pi\Z$ such that $\abs{\theta-n\pi}\simeq\abs{\sin(\theta)}\sim K,L$ ($0<K\leq L\leq1$, $n\in\Z$), respectively.
Then for $K<L/2$ (i.e. $K\ll L$), we have
\begin{align} \label{bilinear-Strichartz-est-homo}
\nm{\big(e^{\ui t\DD}a_K\big)\big(e^{\ui t\DD}b_L\big)}_{L_t^2\l^2_n(\R\times\Z)} \ls
L^{-\frac{1}{2}} \nm{a_K}_{\l^2_n(\Z)}\nm{b_L}_{\l^2_n(\Z)} .
\end{align}

Furthermore, given any functions $u,v$ defined on a space-time slab $I\times\Z$
and their (spatial) frequency-localized pieces with the similar notation $u_K,v_L$ for $K,L$ as above, we have
\begin{align} \label{bilinear-Strichartz-est}
\bnm{u_K v_L}_{L_t^2\l^2_n(I\times\Z)} \ls L^{-\frac{1}{2}}
\nm{u_K}_{S_0(I\times\Z)} \nm{v_L}_{S_0(I\times\Z)} ,
\end{align}
where the $S_0$-norm is given by 
\begin{align*}
    \nm{u}_{S_0(I\times\Z)} := \nm{u}_{L_t^\infty\l^2_n(I\times\Z)} + \nm{\big(\ui\dt+\DD\big)u}_{L_t^1\l^2_n(I\times\Z)}.
\end{align*}
\end{lemma}

\begin{remark}
If restricting frequencies away from the inflection points of $\omega_{{\rm d}}(\theta)$ with the projection $\Ps$, we may combine Lemma~\ref{P:loc Strichartz} to generalize the estimate \eqref{bilinear-Strichartz-est} above to the full set of Strichartz pairs $(q,p)$ and $(\tilde q, \tilde p)$ satisfying $2< q, \tilde q,\leq\infty$, $2\leq p, \tilde p\leq\infty$ and $\tfrac{2}{q}+\tfrac{1}{p}=\tfrac{1}{2}=\tfrac{2}{\tilde q}+\tfrac{1}{\tilde p}$;
here 
\eqref{bilinear-Strichartz-est} corresponds to the case $(q,p)=(\tilde q, \tilde p)=(\infty,2)$.

We also make the trivial remark that 
the estimate \eqref{bilinear-Strichartz-est} still holds if we replace the product of the form $uv$ on the left with either $u\overline{v}$, $\overline{u}v$, or $\overline{u}\overline{v}$,
since the $L_t^2\l^2_n$-norm of these products are the same.
\end{remark}

\begin{proof}
We begin by addressing the homogeneous case \eqref{bilinear-Strichartz-est-homo} (which can be viewed as a particular case of \eqref{bilinear-Strichartz-est} with $u(t)=e^{\ui t\DD}a$ and $v(t)=e^{\ui t\DD}b$).
By duality, it is equivalent to show that  
for any $f\in L_t^2\l^2_n(\R\times\Z)$,
\begin{align} \label{pangfufu}
    \Babs{\Big\langle f, \big(e^{\ui t\DD}a_K\big)\big(e^{\ui t\DD}b_L\big)\Big\rangle_{\!L_t^2\l^2_n} }
    \ls L^{-\frac{1}{2}} \nm{a_K}_{\l^2_n(\Z)}\nm{b_L}_{\l^2_n(\Z)} 
    \nm{f}_{L_t^2\l^2_n(\R\times\Z)} .
\end{align}
Using the Parseval identity \eqref{Parseval-paraproduct} and Fubini, we deduce 
\begin{align*}
    \text{\eqref{pangfufu}-LHS} &= 
    \bigg| \int_{\R}\iint_{\T\times\T} \overline{\widehat{f}(t,\theta_1+\theta_2)}\, \Big[\ue^{-4\ui t\sin^2(\frac{\theta_1}{2})} \widehat{a_K}(\theta_1)\Big] \Big[\ue^{-4\ui t\sin^2(\frac{\theta_2}{2})}\widehat{b_L}(\theta_2)\Big] \,\tfrac{\ud\theta_1}{2\pi}\tfrac{\ud\theta_2}{2\pi} \,\ud t \bigg| \\
    &= \bigg| \iint_{\T\times\T} \!\Big\{\int_{\R} \overline{\widehat{f}(t,\theta_1+\theta_2)} \,\ue^{-\ui t \big[4\sin^2(\frac{\theta_1}{2}) + 4\sin^2(\frac{\theta_2}{2})\big]}\ud t\Big\}\, \widehat{a_K}(\theta_1) \widehat{b_L}(\theta_2) \,\tfrac{\ud\theta_1}{2\pi}\tfrac{\ud\theta_2}{2\pi} \bigg| ,
\end{align*}
where the Fourier transform is taken with respect to the spatial variable, and $\T$ denotes the periodic domain $\R/2\pi\Z$.
We then change variables by writing
\begin{align} \label{fu-meng}
    \theta = \theta_1+\theta_2
    \quad\text{ and }\quad 
    \rho 
    = 4\sin^2(\tfrac{\theta_1}{2})+4\sin^2(\tfrac{\theta_2}{2})
    = 4-2\cos(\theta_1)-2\cos(\theta_2),
\end{align}
and let 
\begin{align} \label{feifei}
    g(\rho, \theta ) := \int_{\R} \overline{\widehat{f}(t,\theta)} \,\ue^{-\ui t \rho}\,\ud t
    = \big[\mathscr{F}[\overline{\mathscr{F}_{\!\rm d}f}]\big](\rho,\theta) ,
\end{align}
which reveals an isometry between $L_t^2\l^2_n(\R\times\Z)$ and $L^2_{\rho,\theta}(\R\times\T)$ (via the Fourier transforms in both spatial and temporal variables) with 
$\nm{f}_{L_t^2\l^2_n(\R\times\Z)} \simeq \nm{g}_{L^2_{\rho,\theta}(\R\times\T)}$ by Plancherel \eqref{Plancherel}.
Thus, combining \eqref{pangfufu} through \eqref{feifei} we see that 
the estimate \eqref{bilinear-Strichartz-est-homo} can be recast as  
\begin{align} \label{feimengmeng}
    \bigg| \iint_{\T\times\T} g(\rho, \theta )\, \widehat{a_K}(\theta_1) \widehat{b_L}(\theta_2) \,\ud\theta_1\ud\theta_2 \bigg|
    \ls L^{-\frac{1}{2}} \nm{a_K}_{\l^2_n(\Z)}\nm{b_L}_{\l^2_n(\Z)} 
    \nm{g}_{L^2_{\rho,\theta}(\R\times\T)} 
\end{align}
for any $g\in L^2_{\rho,\theta}(\R\times\T)$.

To show \eqref{feimengmeng}, we compute the absolute value of the Jacobian of the change of variables $(\theta_1,\theta_2)\mapsto(\rho,\theta)$ given by \eqref{fu-meng}: 
\begin{align*}
    \abs{J} = \abs{\frac{\partial(\,\rho\,,\,\theta\,)}{\partial(\theta_1,\theta_2)}}
    = \abs{\det\! \begin{bmatrix}
2\sin(\theta_1) \!&\! 2\sin(\theta_2) \\
1 \!&\! 1
\end{bmatrix}} 
    = 2\babs{\sin(\theta_1)-\sin(\theta_2)} .
\end{align*}
Using this information to find $\abs{J}\sim L$ for $L\gg K$ (on the support of the integrand of the third line below), together with the Cauchy–Schwarz inequality, 
we get
\begin{align*}
    \text{\eqref{feimengmeng}-LHS} &= \abs{\iint g(\rho, \theta )\, \widehat{a_K}(\theta_1) \widehat{b_L}(\theta_2) \abs{J}^{-1} \ud\rho\ud\theta } \\
    &\leq \nm{g}_{L^2_{\rho,\theta}} \left(\iint \babs{\widehat{a_K}(\theta_1)}^2 \babs{\widehat{b_L}(\theta_2)}^2 \abs{J}^{-2} \ud\rho\ud\theta \right)^{\frac{1}{2}} \\
    &= \nm{g}_{L^2_{\rho,\theta}} \left(\iint \babs{\widehat{a_K}(\theta_1)}^2 \babs{\widehat{b_L}(\theta_2)}^2 \abs{J}^{-1} \ud\theta_1\ud\theta_2 \right)^{\frac{1}{2}} \\
    &\ls \nm{g}_{L^2_{\rho,\theta}} L^{-\frac{1}{2}} \|\widehat{a_K}\|_{L^2_\theta} \|\widehat{b_L}\|_{L^2_\theta} ,
\end{align*}
which implies \eqref{feimengmeng} upon reversing Plancherel \eqref{Plancherel}, and hence \eqref{bilinear-Strichartz-est-homo} follows.

We now turn our attention to the inhomogeneous estimate \eqref{bilinear-Strichartz-est}.
For simplicity we set $F:=\big(\ui\dt+\DD\big)u$ and $G:=\big(\ui\dt+\DD\big)v$, 
and then use Duhamel’s formula to write (for any times $t_0,t\in I$)
\begin{align*}
    u(t) &= e^{\ui (t-t_0)\DD}u(t_0) - \ui\int_{t_0}^{t}e^{\ui (t-s)\DD} F(s)\,\ud s ,\\
    v(t) &= e^{\ui (t-t_0)\DD}v(t_0) - \ui\int_{t_0}^{t}e^{\ui (t-s)\DD} G(s)\,\ud s .
\end{align*}
Therefore, after distributing the terms and applying Minkowski’s integral inequality (to the second, the third, and the fourth terms below), we then insert the homogeneous estimate \eqref{bilinear-Strichartz-est-homo} to obtain 
(here we may drop some of the domains of norm when there is no confusion)
\begin{align*}
    \bnm{u_K v_L}_{L_t^2\l^2_n(I\times\Z)} \leq&\,
     \nm{e^{\ui (t-t_0)\DD}u_K(t_0)\, e^{\ui (t-t_0)\DD}v_L(t_0)}_{L_t^2\l^2_n(I\times\Z)} \\
     &\,+ \nm{e^{\ui (t-t_0)\DD}u_K(t_0) \int_{t_0}^{t}e^{\ui (t-s)\DD} G_L(s)\,\ud s}_{L_t^2\l^2_n}
     \!+ \nm{e^{\ui (t-t_0)\DD}v_L(t_0) \int_{t_0}^{t}e^{\ui (t-s)\DD} F_K(s)\,\ud s}_{L_t^2\l^2_n} \\
     &\,+ \nm{\int_{t_0}^{t}e^{\ui (t-s)\DD} F_K(s)\,\ud s \int_{t_0}^{t}e^{\ui (t-s')\DD} G_L(s')\,\ud s'}_{L_t^2\l^2_n} \\
     \leq&\,
     \nm{e^{\ui (t-t_0)\DD}u_K(t_0)\, e^{\ui (t-t_0)\DD}v_L(t_0)}_{L_t^2\l^2_n(\R\times\Z)} \\
     &\,+ \int_{I} \nm{e^{\ui (t-t_0)\DD}u_K(t_0)\, e^{\ui (t-s)\DD} G_L(s)}_{L_t^2\l^2_n}\ud s
     + \int_{I} \nm{e^{\ui (t-t_0)\DD}v_L(t_0)\, e^{\ui (t-s)\DD} F_K(s)}_{L_t^2\l^2_n}\ud s \\
     &\,+ \int_{I}\int_{I} \nm{e^{\ui (t-s)\DD} F_K(s)\, e^{\ui (t-s')\DD} G_L(s')}_{L_t^2\l^2_n} \ud s \,\ud s' \\
     \ls&\; L^{-\frac{1}{2}} \bnm{u_K(t_0)}_{\l^2_n(\Z)}\bnm{v_L(t_0)}_{\l^2_n(\Z)} \\
     &\;+ L^{-\frac{1}{2}} \bnm{u_K(t_0)}_{\l^2_n(\Z)}\nm{G_L}_{L_t^1\l^2_n(I\times\Z)}
     + L^{-\frac{1}{2}} \bnm{v_L(t_0)}_{\l^2_n(\Z)}\nm{F_K}_{L_t^1\l^2_n(I\times\Z)} \\
     &\;+ L^{-\frac{1}{2}} \nm{F_K}_{L_t^1\l^2_n(I\times\Z)} \nm{G_L}_{L_t^1\l^2_n(I\times\Z)} \\
     \leq&\; L^{-\frac{1}{2}} \Big[ \nm{u_K}_{L_t^\infty\l^2_n} + \nm{\big(\ui\dt+\DD\big)u_K}_{L_t^1\l^2_n} \Big] 
     \Big[ \nm{v_L}_{L_t^\infty\l^2_n} + \nm{\big(\ui\dt+\DD\big)v_L}_{L_t^1\l^2_n} \Big] .
\end{align*}
This proves \eqref{bilinear-Strichartz-est} recalling the definition of the $S_0$-norm (with $F,G$ re-substituted).
\end{proof}

\medskip

\bigskip

\section{Conserved Quantities and Well-posedness} \label{Sec:CL&GWP}

While the $1$-d cubic \eqref{1-c.NLS} is a completely integrable system with infinitely many conservation laws, \eqref{1-c.DNLS} does not preserve the complete integrability of its continuum counterpart (as opposed to the integrable discretization Ablowitz--Ladik system \cite{Ablowitz.Ladik1975}). 
Nevertheless, 
\eqref{1-c.DNLS} still enjoys conservation of $\l^2$-mass 
\begin{align}  \label{l^2-mass}
    \MM[\ah(t)] := \sum_{n\in\Z}\babs{\ah_n(t)}^2
    = \bnm{\ah(t)}_{\l_n^2}^2 
\end{align}
and conservation of energy 
\begin{align} 
    \EE[\ah(t)] := \sum_{n\in\Z} \Big\{\babs{\ah_{n+1}(t)-\ah_{n}(t)}^2 \pm \abs{\ah_n(t)}^4 \Big\} .
\end{align}
This can be verified by differentiating in time, using the equation \eqref{1-c.DNLS} (to substitute for $\ui\dt\ah_n$) and the discrete analogue of integration-by-parts technique: 
\begin{align} \label{dt-M}
    \dt \MM[\ah] &= \sum_{n\in\Z} 2\Re\!\big[\overline\ah_n\cdot \dt\ah_n \big]
    = \sum_{n\in\Z} 2\Im\!\big[(\ui\dt\ah_n)\overline\ah_n \big] \\
    &= 2\Im \sum_{n\in\Z} \Big[-\big(\ah_{n+1}-2\ah_n+\ah_{n-1}\big)
    \pm2|\ah_n|^2\ah_n\Big]\overline\ah_n \no\\
    &= 2\Im \sum_{n\in\Z} \Big[\,\babs{\ah_{n+1}-\ah_{n}}^2 \pm2|\ah_n|^4 \Big]  
    = 0. \no
\end{align}
The conservation of $\EE[\ah]$ can be proved in similar fashion. 

To provide an alternative proof of the $\l^2$-mass conservation:
On the Fourier side, the equation \eqref{1-c.DNLS} becomes 
\begin{align} \label{1-c.DNLS-Fourier}
    \dt\widehat{\ah}(\theta) = \ui\big(\ue^{\ui\theta}+\ue^{-\ui\theta}-2 \big)\widehat{\ah}(\theta) \mp \ui2\int_{\theta=\theta_1+\theta_2+\theta_3} \widehat{\ah}(\theta_1) \widehat{\overline{\ah}}(\theta_2) \widehat{\ah}(\theta_3) ,
\end{align}
where the nonlinear term is calculated by using \eqref{convolution-paraproduct} repeatedly.
We may also rewrite the $\l^2$-mass as (by Plancherel or \eqref{Parseval-paraproduct})
\begin{align} \label{l^2-mass-Fourier}
    \MM[\ah]  
    = \int_{-\pi}^{\pi} \widehat{\ah}(\theta)\overline{\widehat{\ah}(\theta)} \,\tfrac{\ud\theta}{2\pi}
    =\int_{-\pi}^{\pi} \widehat{\ah}(\theta)\widehat{\overline\ah}(-\theta) \,\tfrac{\ud\theta}{2\pi}
    = \int_{\theta_1+\theta_2=0} \widehat{\ah}(\theta_1) \widehat{\overline{\ah}}(\theta_2).
\end{align}
Then we apply time derivative to \eqref{l^2-mass-Fourier}, use symmetry and plug in the equation \eqref{1-c.DNLS-Fourier} to find 
\begin{align} \label{dt-M-Fourier}
    \dt \MM[\ah] &= 2\Re \int_{\theta_1+\theta_2=0} \dt\widehat{\ah}(\theta_1) \widehat{\overline{\ah}}(\theta_2) \\
    &= 2\Re \int_{\theta_1+\theta_2=0} \ui\big(\ue^{\ui\theta_1}+\ue^{-\ui\theta_1}-2 \big) \widehat{\ah}(\theta_1) \widehat{\overline{\ah}}(\theta_2) 
    \mp 2\Re \int_{\theta_1+\theta_2+\theta_3+\theta_4=0} \ui2\, \widehat{\ah}(\theta_1) \widehat{\overline{\ah}}(\theta_2) \widehat{\ah}(\theta_3) \widehat{\overline{\ah}}(\theta_4) . \no
\end{align}
Observe that each of the two expressions above is invariant if one takes complex conjugate of the integral, and  makes change of variables $(\theta_1,\theta_2)\mapsto(-\theta_2,-\theta_1)$ or $(\theta_1,\theta_2,\theta_3,\theta_4)\mapsto(-\theta_2,-\theta_1,-\theta_4,-\theta_3)$ respectively.
Thus, the multiplier in the first term may be replaced by $\frac{1}{2}\big[\ui(\ue^{\ui\theta_1}+\ue^{-\ui\theta_1}-2 ) - \ui(\ue^{\ui\theta_2}+\ue^{-\ui\theta_2}-2 ) \big]$, which equals zero on the set where $\theta_1+\theta_2=0$. 
Similarly, we see  the second term cancels as well.
Summarizing, we have found that both terms on the right side of \eqref{dt-M-Fourier} vanish, and this yields $\dt\MM[\ah]\equiv0$.

Correspondingly, the continuum system \eqref{NLS_low}-\eqref{NLS_high} conserves the $L^2$-mass 
\begin{align}
    \vec{M}[\pe,\qe] := \int_\R \big( \abs{\pe}^2 \!, \abs{\qe}^2 \big) \,\ud x,
\end{align}
as well as the energy (supposing the solution has sufficient regularity)
\begin{align}
    E[\pe,\qe] := \int_\R \Big(\abs{\partial_x\pe}^2 - \abs{\partial_x\qe}^2
    \pm\abs{\pe}^4 \pm\abs{\qe}^4 \pm4\abs{\pe}^2\!\abs{\qe}^2 \Big) \,\ud x .
\end{align}

With the mass conservations, we may now address the question of global well-posedness (GWP) for both the discrete model \eqref{1-c.DNLS} in $\l^2$ and the continuum model \eqref{NLS_low}-\eqref{NLS_high} in $L^2$.

In the case of \eqref{1-c.DNLS}, local solutions in $\ell^2(\Z)$ are constructed by Picard's theorem observing that \eqref{1-c.DNLS}-RHS is locally Lipschitz.
Then the $\l^2$-mass conservation guarantees that such solutions can be extended globally via a continuity argument.
As a byproduct,
the uniform $\l^2_n$-bound follows again from the conservation of $\MM[\ah]$ together with the $\l^2_n$-bound \eqref{ah0-h'} on the initial data $\ah(0)$.

\begin{proposition}[GWP for \eqref{1-c.DNLS}: $\l^2_n$-bound] \label{Prop:l^2-GWP}
        Given $\psi_0, \phi_0 \in L^2(\R)$, 
        the evolution \eqref{1-c.DNLS} with initial data \eqref{E:initial data} admits a global solution $\uh\in C_t\ell_n^2(\R\times\Z)$. Moreover,
        \begin{equation} \label{l^2-bdd}
        \MM[\ah(t)] = \nm{\ah(t)}_{\l_n^2}^2 =\nm{\ah(0)}_{\l_n^2}^2
        \leq h \big( \|\psi_0\|_{L^2}^2 + \|\phi_0\|_{L^2}^2 \big)
        \qquad\text{uniformly for $t\in \R$}.
        \end{equation}
\end{proposition}

Regarding \eqref{NLS_low}-\eqref{NLS_high},
similar arguments as in \cite{Tsutsumi1987} suggest that such equations are globally well-posed in $L^2(\R)$:
Local well-posedness is obtained by contraction mapping in Strichartz spaces (which impose additional spacetime bounds) using the estimates recalled in Lemma~\ref{Lem:l.c-Strichartz};
this is then rendered global using conservation of the $L^2$-norm.

\begin{proposition}[GWP for \eqref{NLS_low}-\eqref{NLS_high}: uniqueness] \label{Prop:GWP-NLS}
    If the initial data $(\psi_0,\phi_0)\in L^2(\R)$, then there exists a \emph{unique} global solution $(\psi,\phi)\in \big(C_tL^2_x\cap L_{t,loc}^6L_x^6\big)(\R\times\R)$ to \eqref{NLS_low}-\eqref{NLS_high}. 
\end{proposition}

\begin{remark}
The construction of solutions via contraction mapping in $C_tL^2_x\cap L^6_{t,x}$ also ensures these solutions are unique in this class and the cubic nonlinearities make sense as spacetime distributions.
In particular, the (conditional) uniqueness result will play a crucial part in closing the proof of the main theorem via the compactness-uniqueness argument (see the proof of Theorem~\ref{Thm:main} at the end of Section~\ref{Sec:Convergence}).
\end{remark}

\bigskip

\section{Almost Conservation Law} \label{Sec:Almost_CL}

A key to our later analysis relies on suppression of the high-frequency component of solutions (or the `medium' frequencies of $\ah$ in our case), so as to prevent a dynamical ``low-to-high frequency cascade'' scenario, 
where the `energy' starts off concentrated in low frequencies but moves increasingly to higher frequencies as time progresses. 

To this end, we recall the Fourier truncation operators $\Ps_{\!<\lambda}$ and $\Ps_{\geq\lambda}$ introduced in Definition~\ref{Def:Ps}, 
which help to give insight into how the conserved $\l^2$-mass $\MM[\ah]$ is moved around in frequency space during the DNLS evolution.
Even though the modified $\l^2$-mass $\MM[\Ps_{\!<\lambda}\ah]$ (or $\MM[\Ps_{\geq\lambda}\ah]$) is not exactly conserved, 
it turns out to enjoy an \emph{almost} conservation law, in the sense that 
the growth in time of this quantity is very small
for sufficiently large $\k=\lambda/h\gg 1$ (which marks the transition from low to high frequencies). 
Note that if $\k\rt\infty$ (independent of $h$), then $\Ps_{\!<\k h}$ formally approximates to the identity, for which $\MM[\Ps_{\!<\k h}\ah]$ would be constant. 
(When $\k\geq h^{-1}$ so that $\lambda=\k h\geq 1$, we may set $\Ps_{\!<\lambda}={\rm Id}$.) 
One can then use this almost conserved quantity to generate long-time control of the solution in much the same way that a genuine conservation law can be used to ensure global well-posedness.

\begin{proposition}[Almost conservation of frequency-truncated $\l^2$-mass] \label{Prop:ACL} 
Under the hypotheses of Theorem~\ref{Thm:main}, specifically, 
let $\ah$ be the solution to \eqref{1-c.DNLS} constructed by Proposition~\ref{Prop:l^2-GWP}.
Assume in addition that for any $T>0$ it satisfies 
\begin{align} \label{Strichartz-assumption}
    \nm{\ah}_{L^4_{\tau}\l^\infty_n ([-h^{-2}T, h^{-2}T]\times\Z)} \ls 
    h^{\frac{1}{2}} , 
\end{align}
where we use the shorthand $\tau=h^{-2}t$. 

Then for $\k\gg1$ we have
\begin{align} \label{ACL}
    \sup_{\abs{t}\leq T}\,
    \MM\big[\Ps_{\!<\k h}\ah(h^{-2}t)\big] = \MM\big[\Ps_{\!<\k h}\ah(0)\big] 
    + O(\k^{-\frac{1}{2}+})\cdot  
    \nm{\ah(0)}_{\l_n^2}^2 , 
\end{align}
or equivalently, for all $\abs{t}\leq T$,
\begin{align} \label{ACL'}
    \Big|\MM\big[\Ps_{\!<\k h}\ah(h^{-2}t)\big] - \MM\big[\Ps_{\!<\k h}\ah(0)\big] \Big|
    = \Big|\MM\big[\Ps_{\geq\k h}\ah(h^{-2}t)\big] - \MM\big[\Ps_{\geq\k h}\ah(0)\big] \Big|
    \ls 
    \k^{-\frac{1}{2}+} 
    \nm{\ah(0)}_{\l_n^2}^2 ,
\end{align}
where the implicit constant is independent of $\k,h$.
\end{proposition}

\begin{remark}
This proposition of ``almost conservation law'' will play a key role in proving Proposition~\ref{Prop:Schrodinger-Dynamics} on persistence of frequency-localization and Schr\"odinger-like Strichartz estimates, as well as the ``equicontinuity in space'' property in Proposition~\ref{thm:pre-compactness}.

For the application of this result in the proof of spatial equicontinuity in Subsection~\ref{SubSec:Space-Equicontinuity},
the precise value of the exponent $-\frac{1}{2}+$ in \eqref{ACL},\eqref{ACL'} is not particularly important; any negative exponent would have sufficed there. 
However, for Proposition~\ref{Prop:Schrodinger-Dynamics}, the exponent $-\frac{1}{2}+$ here is directly tied to the power $\frac{1}{4}-$ of $h$ in \eqref{E:suppress}, which in turn is responsible for the closure of the Strichartz estimates \eqref{E:Strichartz} and \eqref{E:Strichartz'}. 
    
Nonetheless, The exponent $-\frac{1}{2}+$ here might not be optimal. 
To improve the power of $\k$, one can try to exploit other conservation laws, modify the (almost) conserved quantities with additional correction terms (to damp out some non-resonant fluctuations), 
replace the sharp Fourier cutoff with a smooth one, and even possibly adopt more refined estimates.
\end{remark}

\begin{proof}
To investigate the behavior in time of $\MM[\Ps_{\!<\k h}\ah(\tau)]$ (for $\tau=h^{-2}t$),
we first compute the time derivative of this quantity:
Following the same strategy as \eqref{dt-M} in deriving the $\l^2$-mass conservation, we deduce  
(For notational simplicity, we shall abbreviate $\Ps_{\!<\k h}$ as $\Ps$, with symbol $\id_{\mathsf G_{\k h}}(\theta)$ denoted by $\chi(\theta)$, and omit the time variable here and in the sequel.)
\begin{align*}
        \partial_{\tau} \MM[\Ps\ah] &= \sum_{n\in\Z} 2\Re\!\big[\overline{(\Ps\ah)}_n\cdot \partial_{\tau}(\Ps\ah)_n \big] = \sum_{n\in\Z} 2\Im\!\big[\ui\partial_{\tau}(\Ps\ah)_n \overline{(\Ps\ah)}_n \big] \\
    &= 2\Im \sum_{n\in\Z} \Big\{-\big[(\Ps\ah)_{n+1}-2(\Ps\ah)_n+(\Ps\ah)_{n-1}\big]
    \pm2\Ps\big(|\ah_n|^2\ah_n\big)\Big\} \overline{(\Ps\ah)}_n \no\\
    &= 2\Im \sum_{n\in\Z} \babs{(\Ps\ah)_{n+1}-(\Ps\ah)_{n}}^2 \pm 4\Im\sum_{n\in\Z} \Ps\big(|\ah_n|^2\ah_n\big) \overline{(\Ps\ah)}_n  \\
        &= \pm 4 \Im \sum_{n\in\Z} \overline{(\Ps\ah)}_n \cdot \Big[\Ps\big(|\ah_n|^2\ah_n \big) - |(\Ps\ah)_n|^2(\Ps\ah)_n \Big] , 
\end{align*}
where we have applied $\Ps$ to \eqref{1-c.DNLS} (to substitute for $\ui\partial_{\tau}(\Ps\ah)_n$) and note that $\Ps\big(|\ah_n|^2\ah_n\big) \overline{(\Ps\ah)}_n$ is not necessarily real, while $|(\Ps\ah)_n|^2(\Ps\ah)_n\overline{(\Ps\ah)}_n = |(\Ps\ah)_n|^4$ is.
By the Parseval formula \eqref{Parseval-paraproduct}, the expression above may be rewritten as a paraproduct on the Fourier side.
We then use the relation $\theta_1+\theta_2+\theta_3=-\theta_4$ and evenness of the function $\chi(\theta)$ to obtain 
\begin{align} \label{dt-M(Pu)-Fourier}
        \partial_{\tau} \MM[\Ps\ah] &= 
        \pm 4\Im \int_{\theta_1+\theta_2+\theta_3+\theta_4=0} \Big[\chi(\theta_1+\theta_2+\theta_3)\chi(\theta_4) - \chi(\theta_1)\chi(\theta_2)\chi(\theta_3)\chi(\theta_4) \Big] \widehat{\ah}(\theta_1) \widehat{\overline{\ah}}(\theta_2) \widehat{\ah}(\theta_3) \widehat{\overline{\ah}}(\theta_4) \\
        &= \pm 4\Im \int_{\theta_1+\theta_2+\theta_3+\theta_4=0} \Big[\chi(\theta_4)^2 - \chi(\theta_1)\chi(\theta_2)\chi(\theta_3)\chi(\theta_4) \Big] \widehat{\ah}(\theta_1) \widehat{\overline{\ah}}(\theta_2) \widehat{\ah}(\theta_3) \widehat{\overline{\ah}}(\theta_4) . \no
\end{align}

Alternatively, a formal imitation of the Fourier proof of $\l^2$-mass conservation (as in \eqref{dt-M-Fourier}) reveals that 
\begin{align} \label{dt-M(Pu)-Fourier'}
    \partial_{\tau} \MM[\Ps\ah] 
    &= 2\Re \int_{\theta_1+\theta_2=0} \tfrac{1}{2}\big[\chi(\theta_1)^2+\chi(\theta_2)^2\big] \partial_{\tau}\widehat{\ah}(\theta_1) \widehat{\overline{\ah}}(\theta_2) \\
    &= \tfrac{1}{2}\Re \int_{\theta_1+\theta_2=0} \ui \big[\chi(\theta_1)^2+\chi(\theta_2)^2\big] \Big[\big(\ue^{\ui\theta_1}+\ue^{-\ui\theta_1}-2 \big) - \big(\ue^{\ui\theta_2}+\ue^{-\ui\theta_2}-2 \big) \Big] \widehat{\ah}(\theta_1) \widehat{\overline{\ah}}(\theta_2) \no \\
    &\quad \mp 2\Re \int_{\theta_1+\theta_2+\theta_3+\theta_4=0} \ui\big[\chi(\theta_1+\theta_2+\theta_3)^2+\chi(\theta_4)^2\big] \widehat{\ah}(\theta_1) \widehat{\overline{\ah}}(\theta_2) \widehat{\ah}(\theta_3) \widehat{\overline{\ah}}(\theta_4) \no \\
    &= \pm 4\Im \int_{\theta_1+\theta_2+\theta_3+\theta_4=0} \chi(\theta_4)^2\, \widehat{\ah}(\theta_1) \widehat{\overline{\ah}}(\theta_2) \widehat{\ah}(\theta_3) \widehat{\overline{\ah}}(\theta_4) . \no
\end{align}
In particular, the term arising from the dispersion (in the second line) cancels since $\big(\ue^{\ui\theta_1}+\ue^{-\ui\theta_1}-2 \big) - \big(\ue^{\ui\theta_2}+\ue^{-\ui\theta_2}-2 \big)=0$ on the line $\theta_1+\theta_2=0$.
We point out that the resulting quadrilinear form on \eqref{dt-M(Pu)-Fourier'}-RHS actually agrees with the expression on \eqref{dt-M(Pu)-Fourier}-RHS because they are both equal to 
\begin{align*}
    \mp \Im \int_{\theta_1+\theta_2+\theta_3+\theta_4=0} \Big[\chi(\theta_1)^2 - \chi(\theta_2)^2 + \chi(\theta_3)^2 - \chi(\theta_4)^2 \Big] \widehat{\ah}(\theta_1) \widehat{\overline{\ah}}(\theta_2) \widehat{\ah}(\theta_3) \widehat{\overline{\ah}}(\theta_4)
\end{align*}
via the symmetrization generated by permutations on the even arguments $\theta_2$ and $\theta_4$, the odd arguments $\theta_1$ and $\theta_3$, as well as the additional symmetry $m(\theta_1,\theta_2,\theta_3,\theta_4)\mapsto -\overline{m}(-\theta_2,-\theta_1,-\theta_4,-\theta_3)$ for the multiplier $m(\vec{\theta})$.
We also remark that in the case $\chi\equiv1$ (which corresponds to $\k>h^{-1}$ so that $\Ps={\rm Id}$), this results in a vanishing multiplier, thus reproducing the Fourier proof of $\MM[\ah]$ conservation (in Section~\ref{Sec:CL&GWP}).

Integrating \eqref{dt-M(Pu)-Fourier} in time and by the fundamental theorem of calculus, 
we find 
(here $\tau=h^{-2}t$ and $\sigma=h^{-2}s$ for $|t|\leq T$ and $0\leq s/t \leq 1$)
\begin{align*} 
    &\MM\big[\Ps\ah(\tau)\big] - \MM\big[\Ps\ah(0)\big] 
    = \int_0^{\tau} \partial_{\tau}\MM\big[\Ps\ah(\sigma)\big]\,\ud\sigma \\
    =& \pm 4\Im \int_0^{\tau} \!\int_{\sum_{i=1}^4\theta_i=0} \Big[\chi(\theta_4)^2 - \chi(\theta_1)\chi(\theta_2)\chi(\theta_3)\chi(\theta_4) \Big] \widehat{\ah}(\sigma,\theta_1) \widehat{\overline{\ah}}(\sigma,\theta_2) \widehat{\ah}(\sigma,\theta_3) \widehat{\overline{\ah}}(\sigma,\theta_4). 
\end{align*}
In view of the conservation of $\MM[\ah]$ (see \eqref{l^2-bdd}) and with the sharp Fourier cutoff $\Ps=\Ps_{\!<\k h}$ as defined in \eqref{Ps-def}, it follows that
\begin{align} \label{int-t-abs}
    &\;\Big|\MM\big[\Ps\ah(\tau)\big] - \MM\big[\Ps\ah(0)\big] \Big| 
    = \Big|\MM\big[\big(1-\Ps\big)\ah(\tau)\big] - \MM\big[\big(1-\Ps\big)\ah(0)\big] \Big| \\
    \ls&\; \bigg| \int_0^{\tau} \!\int_{\sum_{i=1}^4\theta_i=0} \Big[\chi(\theta_4)^2 - \chi(\theta_1)\chi(\theta_2)\chi(\theta_3)\chi(\theta_4) \Big] \widehat{\ah}(\sigma,\theta_1) \widehat{\overline{\ah}}(\sigma,\theta_2) \widehat{\ah}(\sigma,\theta_3) \widehat{\overline{\ah}}(\sigma,\theta_4) \bigg|. \no
\end{align}
It remains for us to estimate the space-time integral above on \eqref{int-t-abs}-RHS.

We now break every copy of $u$ into a sum of dyadic constituents $u_i$ ($i=1,2,3,4$) using Littlewood–Paley type projections, each localized (with a smooth cutoff function) in spatial frequency space to have Fourier support $\abs{\theta_i-n\pi}\simeq\abs{\sin(\theta_i)}\sim N_ih$ ($n\in\Z$, $N_i=2^{k_i}$ for $k_i\in\{0\}\cup\N$ so that $0<N_ih\leq1$).
Note that we do not need to decompose the frequencies $\abs{\sin(\theta)} \leq h$ here since we are considering the case $\k\gg1$.

When we estimate the dyadic shells that constitute the integral in \eqref{int-t-abs}, generally we will first seek a pointwise bound on the multiplier $m(\vec{\theta})$ (for $\vec{\theta}\in\big\{(\theta_1,\theta_2,\theta_3,\theta_4)\in (\R/2\pi\Z)^4:\theta_1+\theta_2+\theta_3+\theta_4=0\big\}$): 
\begin{align}
    |m(\vec{\theta})| :=
    \abs{\chi(\theta_4)^2 - \chi(\theta_1)\chi(\theta_2)\chi(\theta_3)\chi(\theta_4)}
    \ls B(N_1,N_2,N_3,N_4).
\end{align}
We then factor this bound out of the integral piece, with the absolute value not entering the integrand. 
The remaining quadrilinear form (involving only the dyadic pieces of $u_i$) can be estimated by reversing the Parseval identity, followed by using H\"older’s inequality and the bilinear Strichartz estimates. 

Our aim here is to establish the bound with a decay factor of $\k^{-\frac{1}{2}+}(N_1N_2N_3N_4)^{0-}$. Once we have this, we will be able to sum over all the dyadic pieces $u_i$ 
since the bounds decay geometrically in these frequencies $N_i$.
To achieve this goal, our strategy is to analyze different frequency interactions in the paraproduct under restriction $\sum_{i=1}^4\theta_i=0$, exploiting cancellations and smallness from the multiplier $m(\vec{\theta})$ 
as well as the refined Strichartz estimates (depending on the relative sizes of the frequencies involved).

By the symmetry of the multiplier $m(\vec{\theta})$ in $\theta_1,\theta_2,\theta_3$, and the fact that the bilinear Strichartz estimate \eqref{bilinear-Strichartz-est} 
allows complex conjugates on either factor, we may assume (without loss of generality) that $N_1\geq N_2\geq N_3$. Note also that $\sum_{i=1}^4\theta_i=0$ in the integration of \eqref{int-t-abs}-RHS so that $N_4\ls N_1$. Hence, it suffices to obtain a decay factor of $\k^{-\frac{1}{2}+}N_1^{0-}$. 

Observe that the multiplier $m(\vec{\theta})$ vanishes unless 
$\abs{\sin(\theta_1)}\geq\k h$ and $\abs{\sin(\theta_4)}<\k h$.
Thus, the only summands that contribute to \eqref{int-t-abs}-RHS are those where $N_1\gtrsim\k$ and $N_1\gg N_4$. 
For this class of frequency interactions, we pull out the trivial pointwise bound $|m(\vec{\theta})|\leq1$. 
Then we undo the Parseval formula \eqref{Parseval-paraproduct} and use Cauchy-Schwarz to pair $u_1u_4$ and $u_2u_3$ in $L_t^2\l_n^2$: 
\begin{align} \label{I-piece}
    &\quad\; \bigg| \int_0^{\tau} \!\int_{\sum_{i=1}^4\theta_i=0} m(\vec{\theta})\, 
    \widehat{\ah_1}(\sigma,\theta_1) \widehat{\overline{\ah_2}}(\sigma,\theta_2) \widehat{\ah_3}(\sigma,\theta_3) \widehat{\overline{\ah_4}}(\sigma,\theta_4) \bigg|\\
    &\leq \nm{u_1u_4}_{L_{\tau}^2\l^2_n([-h^{-2}T, h^{-2}T]\times\Z)} \nm{u_2u_3}_{L_{\tau}^2\l^2_n([-h^{-2}T, h^{-2}T]\times\Z)}. \no
\end{align}
Applying the bilinear Strichartz estimate \eqref{bilinear-Strichartz-est} in Lemma~\ref{Lem:bilinear-Strichartz} to the first factor above, followed by using the equation \eqref{1-c.DNLS}, H\"older’s inequality, the uniform $\l^2_n$-bound \eqref{l^2-bdd} in Proposition~\ref{Prop:l^2-GWP}, and the assumption of the Strichartz bound \eqref{Strichartz-assumption}, 
we have 
\begin{align} \label{I-piece-1}
    &\quad \nm{u_1u_4}_{L_{\tau}^2\l^2_n([-h^{-2}T, h^{-2}T]\times\Z)} \\
    &\ls (N_1h)^{-\frac{1}{2}}
    \Big[\nm{u}_{L_{\tau}^\infty\l^2_n([-h^{-2}T, h^{-2}T]\times\Z)} + \nm{\big(\ui\dt+\DD\big)u}_{L_{\tau}^1\l^2_n([-h^{-2}T, h^{-2}T]\times\Z)} \Big]^2 \no \\
    &= (N_1h)^{-\frac{1}{2}}
    \Big[\nm{\ah(0)}_{\l_n^2} + \bnm{\CC[\ah]}_{L_{\tau}^1\l^2_n([-h^{-2}T, h^{-2}T]\times\Z)} \Big]^2 \no \\
    &\ls (N_1h)^{-\frac{1}{2}}
    \Big[\nm{\ah(0)}_{\l_n^2} + (h^{-2}T)^{\frac{1}{2}} \nm{u}_{L_{\tau}^\infty\l^2_n([-h^{-2}T, h^{-2}T]\times\Z)} \nm{u}_{L_{\tau}^4\l^\infty_n([-h^{-2}T, h^{-2}T]\times\Z)}^2 \Big]^2 \no \\ 
    &\ls 
    (1+T) N_1^{-\frac{1}{2}} \nm{\ah(0)}_{\l_n^2}. \no
\end{align}
For the second factor of \eqref{I-piece}-RHS, we once again use H\"older, together with \eqref{l^2-bdd} and \eqref{Strichartz-assumption} to bound 
\begin{align} \label{I-piece-2}
    \nm{u_2u_3}_{L_{\tau}^2\l^2_n([-h^{-2}T, h^{-2}T]\times\Z)} 
    &\ls (h^{-2}T)^{\frac{1}{4}} \nm{u}_{L_{\tau}^\infty\l^2_n([-h^{-2}T, h^{-2}T]\times\Z)} \nm{u}_{L_{\tau}^4\l^\infty_n([-h^{-2}T, h^{-2}T]\times\Z)} \\
    &\ls 
    T^{\frac{1}{4}} \nm{\ah(0)}_{\l_n^2}. \no
\end{align}

Collecting \eqref{I-piece}--\eqref{I-piece-2},
we see that 
\begin{align}
    \text{\eqref{I-piece}-LHS} 
    \ls 
    T^{\frac{1}{4}} (1+T)\, N_1^{-\frac{1}{2}} \nm{\ah(0)}_{\l_n^2}^2.
\end{align}
By our assumptions on the $N_i$ and their relative size in comparison to the parameter $\k$, we have the decay factor $N_1^{-\frac{1}{2}} \ls \k^{-\frac{1}{2}+}N_1^{0-} \ls \k^{-\frac{1}{2}+}(N_1N_2N_3N_4)^{0-}$, which allows us to sum over all the $N_i$ to conclude that 
(uniformly for $\abs{t}\leq T$)
\begin{align} \label{ACL-bound}
    \text{\eqref{int-t-abs}-RHS} 
    \ls 
    T^{\frac{1}{4}} (1+T)\, \k^{-\frac{1}{2}+} \nm{\ah(0)}_{\l_n^2}^2.
\end{align}
This yields our desired estimates \eqref{ACL} and \eqref{ACL'}, and hence completes the proof of Proposition~\ref{Prop:ACL}.
\end{proof}

% ...

\bigskip

\section{Persistence of Schr\"odinger-like Dynamics} \label{Sec:Schrodinger-Dynamics}


As discussed earlier in Subsection~\ref{SubSec:Dispersive-Strichartz}, the presence of inflection points unique to the discrete dispersion relation poses a major obstacle to 
sustaining the continuum Schr\"odinger dynamics for the discrete solutions. 
Our remedy to this issue is to  
suppress the influence of the inflection points 
by controlling the dynamical evacuation of $\l^2$-mass among frequencies for a long period of time. 
This frequency localization control is strong enough to enable us to recover the same desired Strichartz estimates as the continuum Schr\"odinger propagator via a bootstrap scheme.

\begin{proposition}
\label{Prop:Schrodinger-Dynamics}
Under the hypotheses of Theorem~\ref{Thm:main}, specifically, 
let $\ah$ be the solution to \eqref{1-c.DNLS} with initial data \eqref{E:initial data}, 
and let $\psi^h,\phi^h$ be the corresponding continuum reconstruction of this solution built via \eqref{psi^h},\eqref{phi^h}.

Then for any $T>0$, we control the frequency-restricted $\l^2$-norm as 
\begin{align} \label{E:suppress}
    \sup_{\abs{\tau}\leq h^{-2}T} \bnm{ \Ps_{\geq\lambda\,}\ah(\tau)}_{\l^{2}_n} \ls_{T} 
    \lambda^{-\frac{1}{4}+} h^{\frac{1}{4}-} \nm{\ah(0)}_{\l^2_{n}},
\end{align}
where $\Ps_{\geq\lambda}:= 1-\Ps_{\!<\lambda}$ is the projection operator defined in Definition~\ref{Def:Ps}.

Moreover, we have Strichartz bounds 
\begin{align} \label{E:Strichartz}
    \nm{\ah}_{L^6_{\tau}\l^6_n([-h^{-2}T, h^{-2}T]\times\Z)} + \nm{\ah}_{L^4_{\tau}\l^\infty_n ([-h^{-2}T, h^{-2}T]\times\Z)} \ls_{T}  \nm{\ah(0)}_{\l^2_n},
\end{align}
as well as 
\begin{align} \label{E:Strichartz'}
\bnm{\ph}_{L^6_{t,x}([-T,T]\times\r)}+ \bnm{\phi^h}_{L^6_{t,x}([-T,T]\times\r)}
\ls_{T} \|\psi_0\|_{L^2} + \|\phi_0\|_{L^2} .
\end{align}
Here all the implicit constants are independent of $h$. 
\end{proposition}

\begin{proof}
First note that by the uniform $\l^2_n$-bound \eqref{l^2-bdd} from Proposition~\ref{Prop:l^2-GWP}, together with Plancherel, H\"older, and Lemma~\ref{Lem:l^p-embedding}, 
we have the naive bounds 
\begin{align} \label{l^2-bdd-P}
\bnm{ \Ps_{\geq\lambda\,}\ah}_{L_{\tau}^\infty \ell_n^2(\R\times\Z)}
\leq \|\ah\|_{L_{\tau}^\infty \ell_n^2(\R\times\Z)} = \|\ah(0)\|_{\ell_n^2}
\end{align}
and
\begin{align} \label{2:16}
&\nm{\ah}_{L^6_{\tau}\l^6_n([-h^{-2}T, h^{-2}T]\times\Z)} + \nm{\ah}_{L^4_{\tau}\l^\infty_n ([-h^{-2}T, h^{-2}T]\times\Z)} \\
\lesssim&\; (h^{-2}T)^{\frac16} \|\ah\|_{L_{\tau}^\infty \ell_n^2(\R\times\Z)} +(h^{-2}T)^{\frac14} \|\ah\|_{L_{\tau}^\infty \ell_n^2(\R\times\Z)} \notag\\
=&\; \bigl[(h^{-2}T)^{\frac16}+(h^{-2}T)^{\frac14} \bigr] \|\ah(0)\|_{\ell_n^2}. \notag
\end{align}
In order to gain the extra positive $h$-power for \eqref{l^2-bdd-P} and eliminate the dependence on $h$ in \eqref{2:16}, 
we will run a bootstrap-continuity argument involving the control of frequency localization (by the almost conservation law proposition) 
combined with Strichartz estimates (exploiting the dispersive effect from the discrete Schr\"odinger propagator $e^{\ui t\DD}$).

With a view to closing the bootstrap argument, we will estimate our solutions $\ah$ in these norms of interest 
on the time interval $[-h^{-2}T_*, h^{-2}T_*]$ with $T_*\leq T$, which may later be chosen small at first and then extended to reach an arbitrarily large $T$ by iteration. 
Note that \eqref{l^2-bdd-P},\eqref{2:16} also guarantee that the quantities being bootstrapped are initially finite.

For clarity, we set (see \eqref{ah0-h'} and \eqref{ah0-h})
\begin{align} \label{Bootstrap_initial}
    \lim_{h\rt0}\, h^{-\frac{1}{2}}\nm{\ah(0)}_{\l_n^2} = \|\psi_0\|_{L^2} + \|\phi_0\|_{L^2} =: C_0 
    \quad\text{so that}\quad \nm{\ah(0)}_{\l_n^2} \sim h^{\frac{1}{2}} C_0 
    \quad\text{for }\, 0<h\ll1 .  
\end{align}
And our bootstrap assumptions (corresponding to \eqref{E:suppress} and \eqref{E:Strichartz} respectively) read 
\begin{align}
    \bnm{\Ps_{\geq\lambda\,}\ah}_{L_{\tau}^\infty \ell_n^2([-h^{-2}T_*, h^{-2}T_*]\times\Z)} 
    &\ls \lambda^{-\frac{1}{4}+} h^{\frac{1}{4}-} h^{\frac{1}{2}} C , \label{Bootstrap_assumption-1} \\
    \nm{\ah}_{L^4_{\tau}\l^\infty_n ([-h^{-2}T_*, h^{-2}T_*]\times\Z)} 
    &\ls h^{\frac{1}{2}} C . \label{Bootstrap_assumption-2}
\end{align}

We start with the frequency-restricted $\l^2$-norm in \eqref{E:suppress}. 
Applying Proposition~\ref{Prop:ACL} with $\k=h^{-1}\lambda$ 
under \eqref{Bootstrap_assumption-2} and \eqref{Bootstrap_initial},
we have 
\begin{align} \label{bootstrap-est-1}
    \sup_{\abs{\tau}\leq h^{-2}T_*} \bnm{ \Ps_{\geq\lambda\,}\ah(\tau)}_{\l^{2}_n} \ls 
    \bnm{\Ps_{\geq\lambda\,}\ah(0)}_{\l^2_n} +
    T_*^{\frac{1}{8}} C_0^{\frac{1}{2}}C^{\frac{1}{2}} \big(1+T_*^{\frac{1}{2}}C^2\big) \lambda^{-\frac{1}{4}+} h^{\frac{1}{4}-} \nm{\ah(0)}_{\l^2_{n}},
\end{align}
where the detailed prefactor of $T_*,C,C_0$ can be traced from \eqref{I-piece-1} and \eqref{I-piece-2}. 

Next, we turn to the Strichartz norms 
in \eqref{E:Strichartz}.
Let $\Ps=\Ps_{\!<\frac34}$ abbreviate the Fourier cutoff recalled in \eqref{Ps-def} and let $\widetilde{\Ps}=\Ps_{\!<\frac1{100}}$ 
denote a narrower version. 
Noticing that $\Ps \CC[\widetilde{\Ps}\ah]= \CC[\widetilde{\Ps}\ah]$ for the cubic nonlinearity $\CC_n[\ah]:=\pm2|\ah_n|^2\ah_n$ in \eqref{1-c.DNLS}, 
we employ the Duhamel representation of the solution and split the resulting linear and nonlinear terms (by the frequency projections) as follows:
(here $\tau=h^{-2}t$ and $\sigma=h^{-2}s$ for $|t|\leq T_*$ and $0\leq s/t \leq 1$)
\begin{align} \label{Strichartz-Duhamel}
\ah(\tau)&=e^{\ui \tau\DD}\ah(0)-\ui\int_0^{\tau}e^{\ui (\tau-\sigma)\DD}\,\CC\bigl[\ah(\sigma)\bigr]\,\ud\sigma\\
&=e^{\ui \tau\DD}\Ps\ah(0) + e^{\ui \tau\DD}(1-\Ps)\ah(0) \no \\
&\quad -\ui\int_0^{\tau} e^{\ui (\tau-\sigma)\DD}\Ps\CC\bigl[\widetilde{\Ps}\ah(\sigma)\bigr]\,\ud\sigma 
-\ui\int_0^{\tau}e^{\ui (\tau-\sigma)\DD}\Bigl\{\CC\bigl[\ah(\sigma)\bigr] -\CC\bigl[\widetilde{\Ps}\ah(\sigma)\bigr] \Bigr\}\,\ud\sigma. \no
\end{align}

For the first and third terms of \eqref{Strichartz-Duhamel}-RHS above, the presence of the frequency truncation $\Ps$ allows us to invoke $e^{\ui t\DD}\Ps \sim e^{\ui t\D}$ and apply directly the linear Strichartz estimates Lemma~\ref{P:loc Strichartz}, noting that our desired $(6,6)$ and $(4,\infty)$ are both continuum Schr\"odinger Strichartz pairs, but not discrete ones.
Here we treat both spacetime norms simultaneously, which are taken over $[-h^{-2}T_*, h^{-2}T_*]\times\Z$ (when bounding the $L^4_{\tau} \ell^\infty_n$ norm) or $[-h^{-2}T, h^{-2}T]\times\Z$ (for $L^6_{\tau} \ell^6_n$ bounds).\footnote{\,To avoid notation cluster, we shall sometimes drop the domain of norm when there is no risk of confusion.} 
To begin with,
\begin{align}\label{3:05}
    \bigl\| e^{\ui \tau\DD}\Ps\ah(0) \bigr\|_{L^6_{\tau}\l^6_n \cap L^4_{\tau} \ell^\infty_n} \ls \|\Ps\ah(0)\|_{\l_n^2} \leq \|\ah(0)\|_{\ell_n^2}.
\end{align}
Additionally, 
choosing $(\tilde q, \tilde p)=(\infty,2)$ in Lemma~\ref{P:loc Strichartz}, followed by using H\"older, \eqref{l^2-bdd} and \eqref{Bootstrap_assumption-2},
we have 
\begin{align} \label{3:07}
&\nm{\int_0^{\tau} e^{\ui (\tau-\sigma)\DD}\Ps\CC\bigl[\widetilde{\Ps}\ah(\sigma)\bigr]\,\ud\sigma}_{L^6_{\tau}\l^6_n\cap L^4_{\tau} \ell^\infty_n([-h^{-2}T_*, h^{-2}T_*]\times\Z)} \\
\ls&\, \bnm{\Ps \CC[\widetilde{\Ps}\ah]}_{L^{1}_{\tau}\l^{2}_n}
\ls \bnm{\CC[\ah]}_{L^{1}_{\tau}\l^{2}_n([-h^{-2}T_*, h^{-2}T_*]\times\Z)} \no\\
\ls&\: (h^{-2}T_*)^{\frac{1}{2}} \nm{\ah}_{L^{\infty}_{\tau}\l_n^2} \nm{\ah}_{L^4_{\tau} \l^\infty_n}^2 
= (h^{-2}T_*)^{\frac{1}{2}} \nm{\ah(0)}_{\l_n^2}\nm{\ah}_{L^4_{\tau} \l^\infty_n}^2 \no\\
\ls&\: (h^{-2}T_*)^{\frac{1}{2}} (h^{\frac{1}{2}}C)^2 \nm{\ah(0)}_{\l_n^2} 
= T_*^{\frac{1}{2}} C^2 \nm{\ah(0)}_{\l_n^2} .\no
\end{align}

Turning to the remaining two terms in the Duhamel expansion \eqref{Strichartz-Duhamel} (which corresponds to $e^{\ui t\DD}(1-\Ps)$ where the dispersion is relatively weak), we first narrow our focus to just the $L^4_{\tau} \ell^\infty_n$ norm.  
Upon using H\"older,
we apply Lemma~\ref{Lem:l.d-Strichartz} instead with $(q,p) = (6,\infty)$ 
and find
\begin{align}\label{4:06}
\bnm{e^{\ui \tau\DD}(1-\Ps)\ah(0)}_{L^4_{\tau}\l^{\infty}_n([-h^{-2}T_*, h^{-2}T_*]\times\Z)}
&\lesssim (h^{-2}T_*)^{\frac{1}{12}} \nm{e^{\ui \tau\DD}(1-\Ps)\ah(0)}_{L^6_{\tau}\l^\infty_n}\\
&\lesssim (h^{-2}T_*)^{\frac{1}{12}} \bnm{(1-\Ps)\ah(0)}_{\l_n^2} \notag\\
&\ls(h^{-2}T_*)^{\frac{1}{12}} \big(h^{\frac{1}{4}-} h^{\frac{1}{2}}C\big) .\no
\end{align}
Proceeding in a parallel fashion, we deduce 
\begin{align}\label{4:08}
&\nm{\int_0^{\tau}e^{\ui (\tau-\sigma)\DD}\Bigl\{\CC\bigl[\ah(\sigma)\bigr] -\CC\bigl[\widetilde{\Ps}\ah(\sigma)\bigr] \Bigr\}\,\ud\sigma}_{L^4_{\tau}\l^{\infty}_n([-h^{-2}T_*, h^{-2}T_*]\times\Z)} \\
\ls&\: (h^{-2}T_*)^{\frac{1}{12}}\nm{\int_0^{\tau}e^{\ui (\tau-\sigma)\DD}\Bigl\{\CC\bigl[\ah(\sigma)\bigr] -\CC\bigl[\widetilde{\Ps}\ah(\sigma)\bigr] \Bigr\}\,\ud\sigma}_{L^6_{\tau}\l^\infty_n}\no\\
\ls&\: (h^{-2}T_*)^{\frac{1}{12}} \bnm{\CC[\ah] -\CC[\widetilde{\Ps}\ah]}_{L^{1}_{\tau}\l^{2}_n}\no\\
\ls&\: (h^{-2}T_*)^{\frac{7}{12}} \bnm{(1-\widetilde{\Ps})\ah}_{L^{\infty}_{\tau}\l_n^2} \nm{\ah}_{L^4_{\tau} \l^\infty_n}^2 \no\\
\ls&\: (h^{-2}T_*)^{\frac{7}{12}} \big(h^{\frac{1}{4}-} h^{\frac{1}{2}}C\big) (h^{\frac{1}{2}}C)^2 .\no
\end{align}
For the last steps of \eqref{4:06} and \eqref{4:08}, we
take advantage of the small redistribution of $(1-\Ps)$-part with \eqref{Bootstrap_assumption-1}
to compensate for the losses in the discrete Strichartz estimates. 

Summarizing \eqref{Strichartz-Duhamel} and \eqref{3:05}--\eqref{4:08}, 
we obtain 
\begin{align} \label{bootstrap-est-2}
\nm{\ah}_{L^4_{\tau}\l^{\infty}_n([-h^{-2}T_*, h^{-2}T_*]\times\Z)} 
\ls \big(1+T_*^{\frac{1}{2}} C^2\big) \nm{\ah(0)}_{\l^2_n}
+ h^{\frac{1}{12}-} \big(T_*^{\frac{1}{12}}+T_*^{\frac{7}{12}}C^2\big) h^{\frac{1}{2}}C .
\end{align}
Combining the (a priori) estimates \eqref{bootstrap-est-2} with \eqref{bootstrap-est-1},
and taking $T_*$ sufficiently small,
a bootstrap argument 
yields  
\begin{align} 
    \sup_{\abs{\tau}\leq h^{-2}T_*} \bnm{ \Ps_{\geq\lambda\,}\ah(\tau)}_{\l^{2}_n} 
    &\ls 
    \bnm{\Ps_{\geq\lambda\,}\ah(0)}_{\l^2_n} +
     \lambda^{-\frac{1}{4}+} h^{\frac{1}{4}-} \nm{\ah(0)}_{\l^2_{n}}, \label{bootstrap-T*-1}\\
    \nm{\ah}_{L^4_{\tau}\l^{\infty}_n([-h^{-2}T_*, h^{-2}T_*]\times\Z)}
    &\ls \nm{\ah(0)}_{\l^2_n} , \label{bootstrap-T*-2}
\end{align}
with the implicit constants here being universal.
Note also that \eqref{bootstrap-est-1} and \eqref{bootstrap-est-2} imply the norms being bootstrapped are continuous in time.

Given any $T>0$, we iterate this argument above $(1+ \!\big[\frac T{T_*}\big])$ many times\footnote{\,The $\l_n^2$-norm conservation guarantees the increment $T_*$ of time for each iteration is uniform.}
and sum up \eqref{bootstrap-T*-1},\eqref{bootstrap-T*-2} in each iteration 
to conclude that 
\begin{align} 
    \sup_{\abs{\tau}\leq h^{-2}T} \bnm{ \Ps_{\geq\lambda\,}\ah(\tau)}_{\l^{2}_n} 
    &\ls_T 
     \lambda^{-\frac{1}{4}+} h^{\frac{1}{4}-} \nm{\ah(0)}_{\l^2_{n}}, \label{bootstrap-T-1}\\
    \nm{\ah}_{L^4_{\tau}\l^{\infty}_n([-h^{-2}T, h^{-2}T]\times\Z)}
    &\ls_T \nm{\ah(0)}_{\l^2_n} . \label{bootstrap-T-2}
\end{align}
In particular, for \eqref{bootstrap-T-1}, our setting of the initial data \eqref{E:initial data},\eqref{E:initial data'} ensures $\bnm{\Ps_{\geq\lambda\,}\ah(0)}_{\l^2_n}=0$ for small enough $0<h\ll1$ and some certain $0<\lambda<1$ independent of $h$.
Note however that in subsequent iterations the corresponding `initial term' does not necessarily vanish, but inherits the bound from prior iterations and so adds up to the dependence on $T$. 

Now, let us move on to the $L^6_{\tau} \ell^6_n$ and $L^6_{t,x}$ bounds, which we will treat utilizing the closed estimates \eqref{bootstrap-T-1},\eqref{bootstrap-T-2} and so need not to argue via bootstrap.  
Henceforth, all spacetime norms will be taken over the full time interval $[-h^{-2}T, h^{-2}T]$ (or $[-T, T]$ for $L^6_{t,x}$-norm). 

From \eqref{3:05} we already have
\begin{align} \label{3:05'}
    \bigl\| e^{\ui \tau\DD}\Ps\ah(0) \bigr\|_{L^6_{\tau}\l^6_n} \ls 
    \|\ah(0)\|_{\ell_n^2}.
\end{align}
For the term in \eqref{3:07}, 
using instead \eqref{bootstrap-T-2} with \eqref{l^2-bdd} 
we immediately see 
\begin{align} \label{3:07'}
&\nm{\int_0^{\tau} e^{\ui (\tau-\sigma)\DD}\Ps\CC\bigl[\widetilde{\Ps}\ah(\sigma)\bigr]\,\ud\sigma}_{L^6_{\tau}\l^6_n} \ls (h^{-2}T)^{\frac{1}{2}} \nm{\ah(0)}_{\l_n^2}\nm{\ah}_{L^4_{\tau} \l^\infty_n}^2
\ls_{T} \nm{\ah(0)}_{\l_n^2} . 
\end{align}
Mimicking \eqref{4:06} with taking $(q,p) = (9,6)$ 
in Lemma~\ref{Lem:l.d-Strichartz}, we find
\begin{align}\label{4:06'}
\bnm{e^{\ui \tau\DD}(1-\Ps)\ah(0)}_{L^6_{\tau}\l^{6}_n}
&\lesssim (h^{-2}T)^{\frac{1}{18}} \nm{e^{\ui \tau\DD}(1-\Ps)\ah(0)}_{L^9_{\tau}\l^6_n}\\
& \lesssim (h^{-2}T)^{\frac{1}{18}} \bnm{(1-\Ps)\ah(0)}_{\l_n^2} \notag\\
&\ls (h^{-2}T)^{\frac{1}{18}} h^{\frac{1}{4}-} \nm{\ah(0)}_{\l^2_{n}}
\ls_{T} \nm{\ah(0)}_{\l^2_{n}} ,\no
\end{align}
where we used \eqref{bootstrap-T-1} for the last line.
Likewise, paralleling \eqref{4:08}, we have 
\begin{align}\label{4:08'}
&\nm{\int_0^{\tau}e^{\ui (\tau-\sigma)\DD}\Bigl\{\CC\bigl[\ah(\sigma)\bigr] -\CC\bigl[\widetilde{\Ps}\ah(\sigma)\bigr] \Bigr\}\,\ud\sigma}_{L^6_{\tau}\l^{6}_n} \\
\ls&\: (h^{-2}T)^{\frac{1}{18}}\nm{\int_0^{\tau}e^{\ui (\tau-\sigma)\DD}\Bigl\{\CC\bigl[\ah(\sigma)\bigr] -\CC\bigl[\widetilde{\Ps}\ah(\sigma)\bigr] \Bigr\}\,\ud\sigma}_{L^9_{\tau}\l^6_n}\no\\
\ls&\: (h^{-2}T)^{\frac{1}{18}} \bnm{\CC[\ah] -\CC[\widetilde{\Ps}\ah]}_{L^{1}_{\tau}\l^{2}_n}\no\\
\ls&\: (h^{-2}T)^{\frac{5}{9}} \bnm{(1-\widetilde{\Ps})\ah}_{L^{\infty}_{\tau}\l_n^2} \nm{\ah}_{L^4_{\tau} \l^\infty_n}^2 \no\\
\ls&_{T}\: (h^{-2}T)^{\frac{5}{9}} h^{\frac{1}{4}-} \nm{\ah(0)}_{\l^2_{n}} \nm{\ah(0)}_{\l^2_{n}}^2 
\ls_{T} \nm{\ah(0)}_{\l^2_{n}} .\no
\end{align}
The last line here was an application of \eqref{bootstrap-T-1},\eqref{bootstrap-T-2}, and \eqref{l^2-bdd}.

Inserting \eqref{3:05'}--\eqref{4:08'} back into \eqref{Strichartz-Duhamel} gives 
\begin{align} \label{L6l6}
    \nm{\ah}_{L^6_{\tau}\l^{6}_n([-h^{-2}T, h^{-2}T]\times\Z)}
    \ls_T \nm{\ah(0)}_{\l^2_n} .
\end{align}
Finally, in view of \eqref{lqp-Lqp} from Lemma~\ref{Lem:norm-relation} 
and \eqref{ah0-h'}, 
the above yields 
\begin{align} \label{E:Strichartz''}
\bnm{\ph}_{L^6_{t,x}([-T,T]\times\r)}+ \bnm{\phi^h}_{L^6_{t,x}([-T,T]\times\r)}
&\simeq h^{-\frac{1}{2}} \nm{\ah}_{L^6_{\tau}\l^{6}_n([-h^{-2}T, h^{-2}T]\times\Z)} \\
&\ls_{T} h^{-\frac{1}{2}} \nm{\ah(0)}_{\l^2_n}
\ls \|\psi_0\|_{L^2} + \|\phi_0\|_{L^2} , \no
\end{align}
which together with \eqref{bootstrap-T-1},\eqref{bootstrap-T-2},\eqref{L6l6} completes the proof of Proposition~\ref{Prop:Schrodinger-Dynamics}.
\end{proof}

\begin{remark}
For the case of single frequency component, we may alternatively use the conservation of $\EE[\ah]$ to obtain the frequency localization control uniformly for all time,
and then in turn use this result to close the proof of Strichartz bounds.
\end{remark}

\bigskip

\section{Precompactness}

This section is dedicated to the proof of the following precompactness result:

\begin{proposition}[Precompactness in $C_tL_x^2$]\label{thm:pre-compactness}
Under the hypotheses of Theorem~\ref{Thm:main}, we have 
for $T>0$ fixed, the two families of functions $\Psi:=\big\{\ph(t,x): 0<h\leq h_0\big\}$ and $\Phi:=\big\{\qh(t,x): 0<h\leq h_0\big\}$ are precompact in $C([-T,T];L^2_x(\R))$.
\end{proposition}

\begin{proof}
    % ...
By an $L^2$ analogue of the Arzel\`a--Ascoli theorem due to M. Riesz \cite{Riesz1933}, precompactness of the two families 
is equivalent to establishing
uniform boundedness, equicontinuity, and tightness properties. Specifically, we will demonstrate:

\noindent\underline{\emph{Uniform boundedness}}: 
There exists $C>0$ such that
\begin{align}\label{unif bdd=}
\sup_{0<h\leq h_0} \Big\{\bnm{\ph}_{L^\infty_t L^2_x([-T,T]\times\R)} + \bnm{\qh}_{L^\infty_t L^2_x([-T,T]\times\R)} \Big\} \leq C.
\end{align}
\noindent\underline{\emph{Equicontinuity}}: 
For any $\ep>0$, there exists $\delta>0$ so that whenever $\abs{s}+\abs{y}<\delta$,
\begin{align} \label{equi}
\nm{\ph(t+s,x+y)-\ph(t,x)}_{L^2_x} + \nm{\qh(t+s,x+y)-\qh(t,x)}_{L^2_x} < \ep
\end{align}
uniformly for $t\in[-T,T]$ (with $t+s\in [-T,T]$) and for $0<h\leq h_0$.\\[2mm]
\noindent\underline{\emph{Tightness}}: 
For any $\ep>0$, there exists $R>0$ such that
\begin{align} \label{tightness}
\sup_{0<h\leq h_0} \sup_{|t|\leq T}\, \int_{\abs{x}\geq R} \Big[\,\babs{\ph(t,x)}^2 + \babs{\qh(t,x)}^2 \Big] \,\ud x<\ep.
\end{align}


\begin{remark}
Note that all these estimates above 
should be uniform for $t\in[-T,T]$ and $0<h\leq h_0$.
Nevertheless, it suffices to prove these properties 
only for very small $h\in(0,h_1]$ where $h_1$ may depend on $\ep$, as $h\mapsto(\ph,\qh)$ defines a continuous mapping from $(0,h_0]$ to $C([-T,T];L^2_x(\r))$ (in view of the well-posedness of \eqref{1-c.DNLS} 
discussed in Proposition~\ref{Prop:l^2-GWP}) 
and thus precompactness and so the three conditions automatically hold on any closed interval of the form $[h_1,h_0]$. 
\end{remark}

Of the three properties above, uniform boundedness \eqref{unif bdd=} follows readily from \eqref{l^2-bdd} in Proposition~\ref{Prop:l^2-GWP} combined with \eqref{lqp-Lqp} by Lemma~\ref{Lem:norm-relation}: 
\begin{align} \label{unif bdd==}
\bnm{\ph}_{L^\infty_t L^2_x([-T,T]\times\R)} + \bnm{\qh}_{L^\infty_t L^2_x([-T,T]\times\R)} 
\simeq h^{-\frac12} \|\ah\|_{L_{\tau}^\infty \ell_n^2([-h^{-2}T, h^{-2}T]\times\Z)} \ls  \|\psi_0\|_{L^2} + \|\phi_0\|_{L^2},
\end{align}
which is uniformly for $0<h\leq h_0$. 

Next, we will demonstrate the equicontinuity statement \eqref{equi} by treating separately the space and time variables in the following two subsections. 
Finally, we will prove the tightness property \eqref{tightness} in Subsection~\ref{SubSec:Tightness}.

\medskip

\subsection{Equicontinuity in Space} \label{SubSec:Space-Equicontinuity}

We aim to show the equicontinuity property in the spatial variable, namely that for any $\ep>0$, there exists $\delta>0$ so that whenever $\abs{y}<\delta$,
\begin{align} \label{equi-space}
\bnm{\ph(t,x+y)-\ph(t,x)}_{L^2_x} + \bnm{\qh(t,x+y)-\qh(t,x)}_{L^2_x} < \ep
\end{align}
uniformly for $t\in[-T,T]$ and $0<h\leq h_0$.

Firstly, we claim that the above is equivalent to tightness on the Fourier side:\footnote{\,Though only the backward implication is needed for this proof.}
\begin{align} \label{equi-space-2}
\sup_{0<h\leq h_0}  \sup_{|t|\leq T} 
\Big\{\big\|P_{|\xi|\geq \kappa} \,\psi^h(t)\big\|_{L_x^2} + \big\|P_{|\xi|\geq \kappa} \,\phi^h(t)\big\|_{L_x^2} \Big\} \rt0 \quad\text{as }\, \k\rt\infty ,
\end{align}
where $P_{|\xi|\geq \kappa}$ denotes the sharp projection to frequencies $|\xi|\geq \kappa$.

Indeed, by the Plancherel identity \eqref{Plancherel}, 
we may write 
\begin{align} \label{douding}
&\bnm{\ph(t,x+y)-\ph(t,x)}_{L^2_x}^2 + \bnm{\qh(t,x+y)-\qh(t,x)}_{L^2_x}^2 \\ 
=& \int_{\R} \babs{\ue^{\ui y\xi}-1}^2 \Big[\babs{\wph(t,\xi)}^2 + \babs{\wqh(t,\xi)}^2\Big] \tfrac{\ud\xi}{2\pi}
= \int_{\R} 4\sin^2\!\big(\tfrac{y\xi}{2}\big) \Big[\babs{\wph(t,\xi)}^2 + \babs{\wqh(t,\xi)}^2\Big] \tfrac{\ud\xi}{2\pi} \no\\
\ls& \int_{|\xi|< \kappa} \abs{y\xi}^2 \Big[\babs{\wph(t,\xi)}^2 + \babs{\wqh(t,\xi)}^2\Big] \tfrac{\ud\xi}{2\pi}
+ \int_{|\xi|\geq \kappa} \Big[\babs{\wph(t,\xi)}^2 + \babs{\wqh(t,\xi)}^2\Big] \tfrac{\ud\xi}{2\pi} \no\\
\leq&\, \abs{y}^2\!\k^2 \Big[\big\|\psi^h(t)\big\|_{L_x^2}^2 + \big\|\phi^h(t)\big\|_{L_x^2}^2 \Big] 
+ \Big\{\big\|P_{|\xi|\geq \kappa} \,\psi^h(t)\big\|_{L_x^2}^2 + \big\|P_{|\xi|\geq \kappa} \,\phi^h(t)\big\|_{L_x^2}^2 \Big\} . \no
\end{align}
If assuming \eqref{equi-space-2} and considering the uniform boundedness property \eqref{unif bdd=}, for any $\ep>0$ we may first choose $\k=\k(\ep)$ sufficiently large (independently of $h,t$) so that the second term above (corresponding to the high-frequency part) is small, e.g. $<\ep^2\!/4$, 
and then take $\abs{y}<\delta=\delta(\ep)$ sufficiently small (independently of $h,t$) so that the first term above (corresponding to the low-frequency part) is small, e.g. $<\ep^2\!/4$ as well. 
This guarantees that \eqref{equi-space} holds uniformly for $t\in[-T,T]$ and $0<h\leq h_0$. 

For the converse, we notice that 
\begin{align*}
    \int_{\R} \babs{\ue^{\ui y\xi}-1}^2 \k\ue^{-2\k\abs{y}}\,\ud y 
    = \frac{2\xi^2}{\xi^2+4\k^2} \gs 1-\mathds{1}_{[-\k,\k]}(\xi) .
\end{align*}
Integrating the above against $\Big[\babs{\wph(t,\xi)}^2 + \babs{\wqh(t,\xi)}^2\Big]$ in $\xi$, followed by using Fubini and Plancherel (recalling the equalities used in \eqref{douding}),
we get 
\begin{align} \label{mengni}
    &\int_{\R} \Big[\bnm{\ph(t,x+y)-\ph(t,x)}_{L^2_x}^2 + \bnm{\qh(t,x+y)-\qh(t,x)}_{L^2_x}^2 \Big] \k\ue^{-2\k\abs{y}}\,\ud y \\
    &\gs \big\|P_{|\xi|\geq \kappa} \,\psi^h(t)\big\|_{L_x^2}^2 + \big\|P_{|\xi|\geq \kappa} \,\phi^h(t)\big\|_{L_x^2}^2 . \no
\end{align} 
The equicontinuity assertion \eqref{equi-space} implies that for any $\ep>0$ (arbitrarily small), there exists $\delta=\delta(\ep)>0$ (independent of $h,t$) so that we may split the integral on the LHS of \eqref{mengni} and bound the first part as 
\begin{align*}
    \int_{\abs{y}<\delta} \Big[\bnm{\ph(t,x+y)-\ph(t,x)}_{L^2_x}^2 + \bnm{\qh(t,x+y)-\qh(t,x)}_{L^2_x}^2 \Big] \k\ue^{-2\k\abs{y}}\,\ud y 
    \ls \ep^2 \!\int_{\R} \k\ue^{-2\k\abs{y}}\,\ud y 
    \ls \ep^2 ,
\end{align*}
while for the remaining portion of the integral, using again the uniform $L^2$ bound \eqref{unif bdd=} we see that 
\begin{align*}
    &\int_{\abs{y}\geq\delta} \Big[\bnm{\ph(t,x+y)-\ph(t,x)}_{L^2_x}^2 + \bnm{\qh(t,x+y)-\qh(t,x)}_{L^2_x}^2 \Big] \k\ue^{-2\k\abs{y}}\,\ud y \\
    &\ls \sup_{0<h\leq h_0}\! \Big\{\bnm{\ph}_{L^\infty_t L^2_x}^2 + \bnm{\qh}_{L^\infty_t L^2_x}^2 \Big\}\cdot\! \int_{\abs{y}\geq\delta} \k\ue^{-2\k\abs{y}}\,\ud y
    \ls \ue^{-2\k\delta} \rt0 \quad\text{as }\, \k\rt\infty 
\end{align*}
uniformly in $h$ and $t$,
and hence \eqref{equi-space-2} follows.

Now turning our attention to proving \eqref{equi-space-2}, a straightforward computation using \eqref{psi^h'},\eqref{phi^h'} and Plancherel reveals that\footnote{\,In fact, here we may restrict attention to the case $\kappa h<\frac\pi2$ (which allows for the following derivation) since, in the other case when $\kappa h\geq\frac\pi2$, \eqref{pangfu}-LHS is automatically zero due to the fact that $\supp(\widehat{\ph}),\supp(\widehat{\qh})\subseteq\bigl[-\frac{\pi}{2h},\frac{\pi}{2h}\bigr]$.} 
\begin{align} \label{pangfu}
    \big\|P_{|\xi|\geq \kappa} \,\psi^h(t)\big\|_{L_x^2}^2 + \big\|P_{|\xi|\geq \kappa} \,\phi^h(t)\big\|_{L_x^2}^2
    = \int_{\kappa h\leq |h\xi| \leq \pi - \kappa h} \bigl| \widehat{\ah}(h^{-2}t, h\xi) \bigr|^2 \tfrac{\ud\xi}{2\pi}
    = h^{-1} \MM\big[\Ps_{\geq\sin(\k h)}\ah(h^{-2}t)\big] ,
\end{align}
where $\MM[\Ps_{\!\geq\sin(\k h)}\ah]$ is the frequency-truncated $\l^2$-mass for the projection $\Ps_{\geq\lambda}$ with $\lambda=\sin(\k h)$ as defined in Definition~\ref{Def:Ps}.
Thanks to the almost conservation of $\MM[\Ps_{\geq\lambda}\ah]$ from Proposition~\ref{Prop:ACL} (combined with \eqref{E:Strichartz} and \eqref{ah0-h'} so that \eqref{Strichartz-assumption} is satisfied), 
we have 
\begin{align} \label{feimeng}
     \sup_{|t|\leq T}\, \MM\big[\Ps_{\geq\sin(\k h)}\ah(h^{-2}t)\big]
     &\ls_{T}  \MM\big[\Ps_{\geq\sin(\k h)}\ah(0)\big]
     + \big[\tfrac{\sin(\k h)}{h} \big]^{-\frac{1}{2}+} \nm{\ah(0)}_{\l_n^2}^2. 
\end{align}
Combining \eqref{pangfu} and \eqref{feimeng} with noting that $\sin(\theta)\simeq\theta$ for $\abs{\theta}\leq\frac{\pi}{2}$,
followed by using \eqref{ph,qh_0} and \eqref{ah0-h'},
we thus deduce 
\begin{align*}
    &\,\sup_{|t|\leq T} 
\Big\{\big\|P_{|\xi|\geq \kappa} \,\psi^h(t)\big\|_{L_x^2}^2 + \big\|P_{|\xi|\geq \kappa} \,\phi^h(t)\big\|_{L_x^2}^2 \Big\} \\
\ls&_{T\,} \big\|P_{|\xi|\geq \kappa} P_{\leq N} \psi_0 \big\|_{L_x^2}^2 + \big\|P_{|\xi|\geq \kappa} P_{\leq N} \phi_0 \big\|_{L_x^2}^2
+ h^{-1} \k^{-\frac{1}{2}+} h \big( \|\psi_0\|_{L^2}^2 + \|\phi_0\|_{L^2}^2 \big) \\
\leq&\; \big\|P_{|\xi|\geq \kappa} \psi_0 \big\|_{L_x^2}^2 + \big\|P_{|\xi|\geq \kappa} \phi_0 \big\|_{L_x^2}^2 
     + \k^{-\frac{1}{2}+} \big( \|\psi_0\|_{L^2}^2 + \|\phi_0\|_{L^2}^2 \big) ,
\end{align*}
here all the implicit constants are independent of $h$ and the RHS above is also free of $h$.
Therefore, \eqref{equi-space-2} now follows from the dominated convergence theorem,
and hence the question of spatial equicontinuity is settled.

\medskip

\subsection{Equicontinuity in Time}

We now turn to equicontinuity in the time variable, namely, 
for any $\ep>0$, there exists $\delta>0$ so that whenever $\abs{s}<\delta$,
\begin{align} \label{equi-time}
\bnm{\ph(t+s,x)-\ph(t,x)}_{L^2_x} + \bnm{\qh(t+s,x)-\qh(t,x)}_{L^2_x} < \ep
\end{align}
uniformly for $t\in[-T,T]$ (with $t+s\in [-T,T]$) and $0<h\leq h_0$.

By Plancherel and the reconstruction formulas \eqref{psi^h'},\eqref{phi^h'},
we may translate this expression of $\big(\ph,\qh\big)$ into the estimate on our discrete solution $\ah$:
\begin{align} \label{equi-time-2}
&\;\bnm{\ph(t+s,x)-\ph(t,x)}_{L^2_x}^2 + \bnm{\qh(t+s,x)-\qh(t,x)}_{L^2_x}^2 \\
\simeq&\; h^{-1} \Bigl\| \mathds{1}_{|\theta|<\frac\pi2}\Bigl[\wah \big(h^{-2}(t+s),\theta\big)-\wah \big(h^{-2}t,\theta\big)\Bigr]\Bigr\|_{L_{\theta}^2}^2 \no\\
&\;+ h^{-1} \Bigl\| \mathds{1}_{|\theta-\pi|<\frac\pi2}\Bigl[e^{4\ui h^{-2}s} \,\wah \big(h^{-2}(t+s),\theta\big)- \wah \big(h^{-2}t,\theta\big)\Bigr]\Bigr\|_{L_{\theta}^2}^2 \no\\
\ls&\; h^{-1} \Bigl\| \mathds{1}_{|\theta|<\frac\pi2} \bigl[e^{-4 \ui h^{-2}s \sin^2(\frac{\theta}{2})} - 1 \bigr] \wah \big(h^{-2}t,\theta\big) \Bigr\|_{L_{\theta}^2}^2
+ h^{-1} \Bigl\| \mathds{1}_{|\theta-\pi|<\frac\pi2} \bigl[ e^{4 \ui h^{-2}s \cos^2(\frac{\theta}{2})} - 1 \bigr] \wah \big(h^{-2}t,\theta\big) \Bigr\|_{L_{\theta}^2}^2 \no\\
&\; + h^{-1} \bigg\| \int_{h^{-2}t}^{h^{-2}(t+s)}\! e^{-4\ui [h^{-2}(t+s)-\sigma] \sin^2(\frac{\theta}{2})\,} \widehat{\CC[\ah]}(\sigma,\theta)\,\ud\sigma \bigg\|_{L_{\theta}^2}^2 . \no
\end{align}
Here in the last step we represent $\wah \big(h^{-2}(t+s),\theta\big)$ using Duhamel’s formula for the \eqref{1-c.DNLS} evolutions (on the Fourier side) 
\begin{align*}
    \wah \big(h^{-2}(t+s),\theta\big) = e^{-4 \ui h^{-2}s \sin^2(\frac{\theta}{2})\,} \wah \big(h^{-2}t,\theta\big) 
    - \ui \int_{h^{-2}t}^{h^{-2}(t+s)}\! e^{-4\ui [h^{-2}(t+s)-\sigma] \sin^2(\frac{\theta}{2})\,} \widehat{\CC[\ah]}(\sigma,\theta)\,\ud\sigma ,
\end{align*}
where $\CC_n[\ah]:=\pm2|\ah_n|^2\ah_n$,
and then regroup the resulting expression into the linear and nonlinear parts.

To estimate the first two linear terms on \eqref{equi-time-2}-RHS, for $\k \gg 1$ to be chosen shortly, we evaluate the contributions of the regions $|\theta|< \kappa h$, $|\theta-\pi|< \kappa h$, and $\kappa h\leq|\theta|\leq\pi-\kappa h$ (mod $2\pi$) separately: 
\begin{align} \label{equi-time-lin}
&\Bigl\| \mathds{1}_{|\theta|<\frac\pi2} \bigl[e^{-4 \ui h^{-2}s \sin^2(\frac{\theta}{2})} - 1 \bigr] \wah \big(h^{-2}t,\theta\big) \Bigr\|_{L_{\theta}^2}^2
+ \Bigl\| \mathds{1}_{|\theta-\pi|<\frac\pi2} \bigl[ e^{4 \ui h^{-2}s \cos^2(\frac{\theta}{2})} - 1 \bigr] \wah \big(h^{-2}t,\theta\big) \Bigr\|_{L_{\theta}^2}^2 \\
\ls& \int_{\abs{\theta}<\k h} \bigl| e^{-4 \ui h^{-2}s \sin^2(\frac{\theta}{2})} - 1 \bigr|^2 \bigl| \wah \big(h^{-2}t,\theta\big)\bigr|^2 \,\tfrac{\ud\theta}{2\pi}
+ \int_{\abs{\theta-\pi}<\k h} \bigl| e^{4 \ui h^{-2}s \cos^2(\frac{\theta}{2})} - 1 \bigr|^2 \bigl| \wah \big(h^{-2}t,\theta\big)\bigr|^2 \,\tfrac{\ud\theta}{2\pi} \no\\
&+ \int_{\kappa h\leq|\theta|\leq\pi-\kappa h} \bigl| \wah \big(h^{-2}t,\theta\big)\bigr|^2 \,\tfrac{\ud\theta}{2\pi} \no\\
\ls&\, \abs{s}^2\!\k^4 \bigl\|\ah(h^{-2}t)\bigr\|_{\ell_n^2}^2 
+ h \Big[\big\|P_{|\xi|\geq \kappa} \,\psi^h(t)\big\|_{L_x^2}^2 + \big\|P_{|\xi|\geq \kappa} \,\phi^h(t)\big\|_{L_x^2}^2 \Big]
\ls \ep^2 h . \no
\end{align}
The last inequality above is guaranteed by choosing $\kappa=\kappa(\ep)$ sufficiently large in \eqref{equi-space-2}, the equivalent of spatial equicontinuity providing the high-frequency control, followed by taking $|s|<\delta=\delta(\ep)$ sufficiently small along with the uniform $\l^2$-bound \eqref{l^2-bdd}.

We then estimate the nonlinear contribution of \eqref{equi-time-2}-RHS 
using Plancherel, Minkowski's integral inequality (or alternatively, Lemma~\ref{Lem:l.d-Strichartz} with $(q,p)=(\tilde q, \tilde p)=(\infty,2)$) and H\"older's inequality, followed by \eqref{E:Strichartz} in Proposition~\ref{Prop:Schrodinger-Dynamics} and \eqref{l^2-bdd} in Proposition~\ref{Prop:l^2-GWP}:
\begin{align} \label{equi-time-non}
&\: \bigg\| \int_{h^{-2}t}^{h^{-2}(t+s)}\! e^{-4\ui [h^{-2}(t+s)-\sigma] \sin^2(\frac{\theta}{2})\,} \widehat{\CC[\ah]}(\sigma,\theta)\,\ud\sigma \bigg\|_{L_{\theta}^2}^2 \\
\ls&\: \bnm{\CC[\ah]}_{L^{1}_{\tau}\l^{2}_n([h^{-2}t,h^{-2}(t+s)]\times\Z)}^2 
\ls \Big[\big(h^{-2}|s|\big)^{\frac{1}{2}} \nm{\ah}_{L^6_{\tau}\l^6_n([-h^{-2}T, h^{-2}T]\times\Z)}^3 \Big]^2 \no\\
\ls&_{T}\; h^{-2}|s| \nm{\ah(0)}_{\l_n^2}^6
\ls \ep^2 h , \no
\end{align}
provided $|s|<\delta=\delta(\ep)$ is taken sufficiently small.

Finally, collecting \eqref{equi-time-2} and \eqref{equi-time-lin},\eqref{equi-time-non} settles \eqref{equi-time}.
Besides, the fact that all the implicit constants in our estimates above are independent of $t,h$ ensures the uniformity for $t\in[-T,T]$ and $0<h\leq h_0$.

\medskip

\subsection{Tightness} \label{SubSec:Tightness}

Lastly, we intend to show that for any $\ep>0$, there exists $R>0$ such that
\begin{align} \label{tightness'}
\sup_{0<h\leq h_0} \sup_{|t|\leq T}\, \int_{\abs{x}\geq R} \Big[\,\babs{\ph(t,x)}^2 + \babs{\qh(t,x)}^2 \Big] \,\ud x<\ep.
\end{align}

We want to reduce this tightness property of $\big(\ph,\qh\big)$ to that of the orbit of solutions $\ah_n(t)$ to \eqref{1-c.DNLS}.
To this end, 
let $\varphi(\cdot)$ be a standard smooth bump function, 
from which we build a cutoff function on the lattice via $\varphi_R(n):=1-\varphi\bigl(\tfrac{nh}{R}\bigr)$.
We will also be decomposing in frequency, and so define a frequency projection $\widetilde{\Ps}$ by 
$\widehat{\widetilde{\Ps}\ah} (\theta) := \bigl[\varphi\bigl(\tfrac\theta{\kappa h}\bigr) + \varphi\bigl(\tfrac{\theta-\pi}{\kappa h}\bigr)\bigr] \widehat\ah(\theta)$, 
which is a smooth variant of $\Ps_{\!<\k h}$ (see \eqref{Ps-def}).

We first observe that the tightness \eqref{tightness'} of $\big(\ph,\qh\big)$ can be interpreted as 
\begin{align}
    \sup_{0<h\leq h_0} \sup_{|\tau|\leq h^{-2}T} \bigl\|\varphi_R\,\ah(\tau)\bigr\|_{\ell^2_n}^2<\ep h 
\end{align}
via some locality preservation property and $\l^2\rt L^2$ boundedness of the reconstruction operator $\rr$.

Since equicontinuity in space (i.e. tightness on Fourier side) \eqref{equi-space-2} (see also \eqref{pangfu}) already provides `medium'-frequency control of $\big\|(1-\widetilde{\Ps})\ah(\tau)\big\|_{\ell^2_n}$, it now remains to prove 
\begin{align}
    \sup_{0<h\leq h_0} \sup_{|\tau|\leq h^{-2}T} \bigl\|\varphi_R\,\widetilde{\Ps}\ah(\tau)\bigr\|_{\ell^2_n}^2<\ep h .
\end{align}

Next, we basically perform the energy estimate: 
From \eqref{1-c.DNLS}, we have
\begin{align} 
\partial_{\tau} \bigl\|\varphi_R\,\widetilde{\Ps}\ah(\tau)\bigr\|_{\ell^2_n}^2 
= -2\Im \sum_{n\in\Z} \varphi_R^2(n)\, \overline{\widetilde{\Ps}\ah_n(\tau)} \cdot \widetilde{\Ps}\Bigl\{ \bigl(\DD\ah\bigr)_n(\tau)- \CC_n\bigl[\ah(\tau)\bigr] \Bigr\}.
\end{align}
Then we estimate the linear and nonlinear contributions on the RHS above separately,
and finally conclude by Gr\"onwall's inequality and tightness of the initial-data term.
\end{proof}
% ...

\bigskip

\section{Convergence of the Flows} \label{Sec:Convergence}

As a consequence of Proposition~\ref{thm:pre-compactness}, 
every sequence $h_n\rt 0$ admits a subsequence $h_{n_j}\rt 0$ such that $(\pe^{h_{n_j}},\qe^{h_{n_j}})$ converges in $C([-T,T];L^2_x(\r))$ to some $(\pe^*,\qe^*)$.   
The last aspect toward proving our main theorem is to show that 
all such subsequential limits are solutions
to 
\eqref{NLS_low}-\eqref{NLS_high} 
with initial data $(\pe_0,\qe_0)$ and satisfy certain spacetime bounds. 

\begin{proposition}\label{Prop:convergence}
Under the hypotheses of Theorem~\ref{Thm:main}, let $\pe^*,\qe^*\in C([-T,T];L^2_x(\R))$ be such that
\begin{align}\label{convg}
\psi^{h_j} \to \psi^* \qtq{and} \phi^{h_j} \to\phi^* \quad \text{in \,$C([-T,T];L^2_x(\R))$}
\end{align}
along some sequence $h_j\to 0$.

Then $\pe^*,\qe^* \in L^6_{t,x}([-T,T]\times\R)$ and they 
solve \eqref{NLS_low}-\eqref{NLS_high} with initial data $(\pe_0,\qe_0)$ in the sense that 
\begin{align}
    \pe^*(t) &= e^{\ui t\D}\pe_0
    \mp2\ui\int_0^te^{\ui (t-s)\D}\Big[\big(\abs{\pe^*}^2+2\abs{\qe^*}^2\big)\pe^*\Big](s)\,\ud s,\label{duhamel 1}\\
    \phi^*(t) &= e^{-\ui t\D}\phi_0
    \mp2\ui\int_0^te^{-\ui (t-s)\D}\Big[\big(\abs{\qe^*}^2+2\abs{\pe^*}^2\big)\qe^*\Big](s)\,\ud s \label{duhamel 2}
\end{align}
in $C([-T,T];L^2_x(\R))$.
\end{proposition}

\begin{remark}
With Proposition~\ref{Prop:convergence} and
invoking the (conditional) uniqueness of \eqref{NLS_low}-\eqref{NLS_high} solutions constructed in $C([-T,T];L^2_x(\R))\cap L^6_{t,x}([-T,T]\times\R)$ 
by Proposition~\ref{Prop:GWP-NLS}, 
it follows that all subsequential limits agree which coincide with the unique solution $(\pe,\qe)$ to \eqref{NLS_low}-\eqref{NLS_high} with initial data $(\pe_0,\qe_0)$, 
and so the original sequence converges without passing to subsequences at all.
We therefore conclude that $\psi^h \to\psi$ and $\phi^h \to\phi$ in $C([-T,T];L^2_x(\R))$ as $h\to 0$.
For notational simplicity, we will omit the subscripts on $h$ and also simply write $(\pe,\qe)$ for $(\pe^*,\qe^*)$ in the following proof.
\end{remark}

\begin{proof}
As $\psi^h$ and $\phi^h$ converge in $C([-T,T];L^2_x(\R))$ due to Proposition~\ref{thm:pre-compactness}, they converge distributionally on $[-T,T]\times\R$.  Consequently, by the uniform $L^6_{t,x}$-bound \eqref{E:Strichartz'} and weak lower-semicontinuity of the norm, the limits $\psi$ and $\phi$ satisfy
\begin{align} \label{6 bdd}
\|\psi\|_{L^6_{t,x}([-T,T]\times\R)} + \|\phi\|_{L^6_{t,x}([-T,T]\times\R)}\lesssim_T  \|\psi_0\|_{L^2} + \|\phi_0\|_{L^2}.
\end{align}

In order to prove that $(\pe,\qe)$ satisfies \eqref{duhamel 1}-\eqref{duhamel 2}, our starting point is the Duhamel formula satisfied by the solution $\ah$ of \eqref{1-c.DNLS}: For any $|t|\leq T$,
\begin{align}\label{ah-duh}
    \ah_n(h^{-2}t) = 
    e^{\ui h^{-2}t\DD}\ah_n(0) 
    - \ui h^{-2}\!\int_0^{t} e^{\ui h^{-2}(t-s)\DD}\,\CC_n\bigl[\ah(h^{-2}s)\bigr]\,\ud s.
\end{align}
Recalling the relations \eqref{psi^h},\eqref{phi^h} (with the notation introduced in \eqref{R-def}), 
we reconstitute $\psi^h,\phi^h$ from the LHS above and find 
\begin{align}
    \psi^h(t) &= \rr\big[\uh_n(h^{-2}t)\big] \label{psi^h=}\\
    &= \rr\big[ e^{\ui h^{-2}t\DD}\ah_n(0)\big] 
    - \rr\Big\{ \ui h^{-2}\!\int_0^{t} e^{\ui h^{-2}(t-s)\DD}\CC_n\bigl[\ah(h^{-2}s)\bigr]\ud s \Big\} , \no\\
    \phi^h(t) &= e^{4\ui h^{-2}t} \,\rr\big[(-1)^n\uh_n(h^{-2}t)\big] \label{phi^h=}\\
    &= e^{4\ui h^{-2}t} \,\rr\big[(-1)^n e^{\ui h^{-2}t\DD}\ah_n(0) \big]
    - e^{4\ui h^{-2}t} \,\rr\Big\{(-1)^n \ui h^{-2}\!\int_0^{t} e^{\ui h^{-2}(t-s)\DD}\CC_n\bigl[\ah(h^{-2}s)\bigr]\ud s \Big\} . \no
\end{align}
We will establish \eqref{duhamel 1},\eqref{duhamel 2} by taking the limit as $h\rt0$ on both sides of \eqref{psi^h=},\eqref{phi^h=}.
Given that \eqref{convg} already ensures the LHS converge, it thus suffices to demonstrate convergence of the RHS, for which each term will be treated separately. 

For the convergence of the linear terms on the RHS of \eqref{psi^h=} and \eqref{phi^h=}, we have
\begin{align}
&\lim_{h\to 0}\,\Bigl\| \rr\bigl[e^{\ui h^{-2}t\DD}\ah_n(0)\bigr] - e^{\ui t\D}\pe_0   \Bigr\|_{L_t^\infty L_x^2([-T,T]\times\R)}=0,\\
&\lim_{h\to 0}\,\Bigl\|e^{4\ui h^{-2}t}  \,\rr\bigl[(-1)^n e^{\ui h^{-2}t\DD}\ah_n(0)\bigr]- e^{-\ui t\D}\phi_0   \Bigr\|_{L_t^\infty L_x^2([-T,T]\times\R)}=0.
\end{align}

Turning to the convergence of the nonlinear terms on the RHS of \eqref{psi^h=} and \eqref{phi^h=},  we intend to show that
\begin{align}
&\lim_{h\to 0}\, \rr\Big\{ h^{-2}\!\int_0^{t} e^{\ui h^{-2}(t-s)\DD}\CC_n\bigl[\ah(h^{-2}s)\bigr]\ud s \Big\} 
= \pm2\int_0^te^{\ui (t-s)\D}\Big[\big(\abs{\pe}^2+2\abs{\qe}^2\big)\pe\Big](s)\,\ud s , \label{convergence-non-1}\\
&\lim_{h\to 0}\, e^{4\ui h^{-2}t} \,\rr\Big\{(-1)^n h^{-2}\!\int_0^{t} e^{\ui h^{-2}(t-s)\DD}\CC_n\bigl[\ah(h^{-2}s)\bigr]\ud s \Big\} 
= \pm2\int_0^te^{-\ui (t-s)\D}\Big[\big(\abs{\qe}^2+2\abs{\pe}^2\big)\qe\Big](s)\,\ud s \label{convergence-non-2}
\end{align}
in $C([-T,T];L^2_x(\R))$ sense.

To provide a heuristic derivation of the limit continuum nonlinearities,
upon substituting with the relation (from \eqref{u_n-ph,qh})
\begin{align*} 
    \uh_n(h^{-2}t) = h \big[\ph(t,hn)+e^{-4\ui h^{-2}t}(-1)^n \qh(t,hn) \big],
\end{align*}
we expand out the cubic nonlinearity:
\begin{align}
    \CC_n\bigl[\ah(h^{-2}t)\bigr] :=&\; 
    \bigl[\pm2|\ah_n|^2\ah_n\bigr] (h^{-2}t) \no\\
    =&\;\pm 2h^3 \bigl|\ph(t,hn)+e^{-4\ui h^{-2}t}(-1)^n\qh(t,hn)\bigr|^2 \bigl[\ph(t,hn)+e^{-4\ui h^{-2}t}(-1)^n\qh(t,hn)\bigr] \no\\
    =&\;\pm2h^3 \Big\{ \Big[\, \babs{\ph}^2\ph + 2\babs{\qh}^2\ph + e^{-8\ui h^{-2}t}\big(\qh\big)^2\overline\ph \,\Big](t,hn) \label{non-expansion-1}\\
    &\qquad\qquad +e^{-4\ui h^{-2}t}(-1)^n \Big[\,\babs{\qh}^2\qh + 2\babs{\ph}^2\qh
    + e^{8\ui h^{-2}t}\big(\ph\big)^2\overline\qh \,\Big](t,hn)\Big\}. \label{non-expansion-2}
\end{align}
By looking at the Fourier supports, we see that only the first row of terms (in \eqref{non-expansion-1}) contributes to \eqref{convergence-non-1}-RHS, while the second row (terms in \eqref{non-expansion-2}) contributes only to \eqref{convergence-non-2}-RHS.
In particular, the two types of mixed interactions with the fast-oscillating factors in time will also drop out of the $h\rt0$ limit in the presence of time integral; 
we can employ a Riemann–Lebesgue-type lemma to justify this temporal non-resonance phenomenon.
These explain why we are eventually led to the coupled nonlinearities as in \eqref{NLS_low}-\eqref{NLS_high}.
\end{proof}

We are finally ready to prove our main result:

\begin{proof}[Proof of Theorem~\ref{Thm:main}]
As noted at the beginning of this section, 
Proposition~\ref{thm:pre-compactness} guarantees that every sequence of $h\rt 0$ admits a subsequence along which both $\pe^{h}$ and $\phi^{h}$ converge in $C([-T,T];L^2_x(\r))$.   
By Proposition~\ref{Prop:convergence}, the limiting functions lie in $\big(C_t L^2_x \cap L^6_{t,x}\big)([-T,T]\times\R)$ and solve the Duhamel formulation \eqref{duhamel 1}-\eqref{duhamel 2} of \eqref{NLS_low}-\eqref{NLS_high}.

On the other hand, by the uniqueness result from Proposition~\ref{Prop:GWP-NLS}, these integral equations admit only one solution in $C_t L^2_x \cap L^6_{t,x}$.
This implies that all subsequential limits are identical, and thus $(\psi^h,\phi^h)$ converges in $C([-T,T];L^2_x(\R))$ as $h\to 0$ without having to pass to a subsequence and the resulting limit is the unique solution $(\pe,\qe)$ to \eqref{NLS_low}-\eqref{NLS_high} with initial data $(\pe_0,\qe_0)$.

Moreover, the relation \eqref{u_n-ph,qh} together with Lemma~\ref{L:sums to int} allows us to convert this convergence \eqref{E:T:main} on the continuum side into the discrete mode of convergence described in \eqref{E:T:main-2}.
\end{proof}

\end{document}